\documentclass[final,11pt,3p]{elsarticle}

\usepackage{graphicx}
\usepackage{amssymb}
\usepackage{amsthm}
\usepackage{amsmath}
\usepackage{algorithm}
\usepackage{algorithmic}
\usepackage{tikz}
\usepackage{lineno}
\usepackage{hyperref}
\usetikzlibrary{shapes}

\usepackage{subcaption}
\usepackage{caption}
\usepackage{float}
\usepackage{geometry}
\usepackage{mathrsfs} 
\usepackage{diagbox}
\biboptions{sort&compress}
\newtheorem{theorem}{Theorem}
\newtheorem{lemma}{Lemma}
\newtheorem{definition}{Definition}

\newtheorem{corollary}{Corollary}

 \newcommand{\Dat}{\mathcal{D}^{\alpha}_{t}}
  \newcommand{\CF}{\mathcal{CF}}
 \newcommand{\CFC}{\mathcal{CFC}}
 
 \newcommand{\cB}{\mathcal{B}}
 \newcommand{\cH}{\mathcal{H}}
 \newcommand{\cJ}{\mathcal{J}}
 \newcommand{\cL}{\mathcal{L}}
 \newcommand{\cN}{\mathcal{N}}
 \newcommand{\cW}{\mathcal{W}}
\usepackage{xcolor}
\begin{document}
	
\begin{tikzpicture}[remember picture,overlay]
	\node[anchor=north east,inner sep=20pt] at (current page.north east)
	{\includegraphics[scale=0.2]{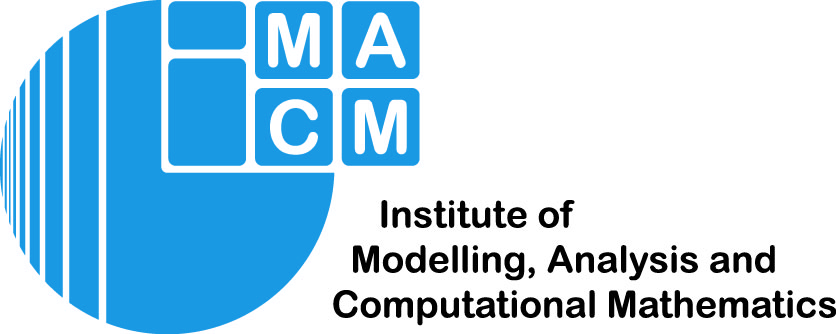}};
\end{tikzpicture}

\begin{frontmatter}

\title{Optimal Control of a Diffusive Epidemiological Model \\ Involving the Caputo-Fabrizio Fractional Time-Derivative}
% old title 
% \title{Optimal Control of a Fractional Epidemiological Model Involving a Spatial Diffusion Operator and the Caputo-Fabrizio Fractional Time-Derivative}

\author[AMNEA]{Achraf Zinihi}
\ead{a.zinihi@edu.umi.ac.ma}

\author[AMNEA]{Moulay Rchid Sidi Ammi}
\ead{rachidsidiammi@yahoo.fr}

\author[BUW]{Matthias Ehrhardt\corref{Corr}}
\cortext[Corr]{Corresponding author}
\ead{ehrhardt@uni-wuppertal.de}

\address[AMNEA]{Department of Mathematics, MAIS Laboratory, AMNEA Group, Faculty of Sciences and Technics,\\
Moulay Ismail University of Meknes, Errachidia 52000, Morocco}

\address[BUW]{University of Wuppertal, Chair of Applied and Computational Mathematics,\\
Gaußstrasse 20, 42119 Wuppertal, Germany}

% ===================================================================
\begin{abstract}
In this work we study a fractional SEIR biological model of a reaction-diffusion, using the non-singular kernel Caputo-Fabrizio fractional derivative in the Caputo sense and employing the Laplacian operator.
In our PDE model, the government seeks immunity through the vaccination program, which is considered a control variable.
Our study aims to identify the ideal control pair that reduces the number of infected/infectious people and the associated vaccine and treatment costs over a limited time and space. 
Moreover, by using the forward-backward algorithm, the approximate results are explained by dynamic graphs to monitor the effectiveness of vaccination.
\end{abstract}


%%%%%%%%%%%%%%%%%%%%%% Journal name
\journal{}

% ====================================================================

\end{frontmatter}

% ==================================================================== Intro
\section{Introduction}\label{S1}
In the literature on calculus, we find that fractional derivatives have a long history of more than 300 years \cite{Machado2011, Oldham1974}. 
It is an old topic that arose as a result of a pertinent question that G.F.A.~de l'H\^opital asked G.W.~Leibniz in a letter about the possible meaning of a derivative of order $\frac{1}{2}$ \cite{Leibniz1849letter}. 
The Scott Blair model of sticky flexible material is the simple application of the fractional derivative \cite{Stiassnie1979}, where stress is equal to the order derivative of strain. 
There are other applications of the fractional derivative, for example, in quantum physics \cite{SahaRay2017}, the anomalous diffusion of particles \cite{Lutz2001}, the chaotic dynamics of the fractional Lorenz system \cite{Grigorenko2003}, and non-diffusive transport in plasma turbulence \cite{delCastilloNegrete2005}. 
The fractional derivative has also been used in mathematics \cite{Abbasbandy2022, SidiAmmi2022}, chemistry \cite{ToledoHernandez2014}, biology \cite{Mokhtari2019, Zinihi2024}, and so on. 
Nevertheless, some fundamental problems have prevented the popularization of fractional calculus. 

For example, in the fractional differential equations with the Riemann-Liouville fractional derivative, the physical meanings of the initial values are unknown \cite{Diethelm2012, Jumarie2007, Gorenflo1999}.
While the Caputo fractional derivative was introduced to avoid this difficulty \cite{Diethelm2012, Gorenflo1999}. 
The initial values in the case of Caputo are similar to those of integer differential equations, so the physical meanings are known. 
However, the kernel of this derivative has a singularity. 
Fractional derivatives with non-singular kernels have attracted more attention and interest from the scientific community. 
The Caputo-Fabrizio fractional time derivative in the Caputo sense ($\CFC$) is sometimes preferred for modeling physical or biological dynamical systems, giving a good description of the phenomena of diffusion and heterogeneity at different scales \cite{Arshad2020, Hikal2021, Maamar2024}.

As is well known, mathematical models have played an essential and very important role in the study of the dynamic evolution of infectious diseases, which has attracted more and more attention from biomathematicians, physicists, medical scientists, etc.\ \cite{Das2017, Guiro2023, SidiAmmi2023}. 
Recently, reaction-diffusion systems have often been used to describe the movement of people and the heterogeneity of space, which play an important role in the spread of some infectious diseases, such as coronavirus, dengue or Ebola \cite{Forna2020, Hamdan2022, Kevrekidis2021}. 

To describe the spread of the disease in a spatial environment, many researchers have extensively investigated different types of spatiotemporal epidemiological models.
In particular, authors in \cite{Chinviriyasit2010}, analytically and numerically investigate the behavior of positive solutions of a reaction-diffusion SIR model for transmission diseases such as pertussis. 
\cite{Kuniya2018} proposed a nonlocal diffusion epidemiological SIR model and obtained threshold theories on the close global stability of disease-free and endemic equilibria by constructing appropriate Lyapunov functions. 
In the research presented by \cite{SidiAmmi2023}, the authors studied a reaction-diffusion SIR epidemic model, expertly formulated as a parabolic system of PDEs, the focus of their investigation is the development of an optimal control strategy, thoughtfully designed to mitigate the spread of infection and reduce the costs associated with vaccination. 
In addition, studies of reaction-diffusion epidemiological systems attract much attention, since the pattern of transmission can be an important indicator of how diseases spread.

In recent years, there has been a growing interest in incorporating fractional calculus and nonlinear operators into mathematical models to provide more accurate and realistic descriptions of complex phenomena. 
When considering the SEIR epidemic model, the use of the $\CFC$ fractional time derivative and the Laplacian operator can provide several motivations. 
For a recent review on epidemiological models, including fractional order and PDE models we refer the reader to \cite{hoang2024differential}.

First, the $\CFC$ fractional time-derivative provides a suitable framework for capturing the memory and inheritance properties often observed in epidemic dynamics. 
Traditional integer-order derivatives assume instantaneous interactions and do not account for the delayed effects of past interactions. 
By incorporating the $\CFC$ fractional time-derivative, the SEIR model can better capture the influence of past states on current dynamics, allowing for a more accurate representation of epidemic spread and a deeper understanding of epidemic behavior, and potentially improving the accuracy of predictions, aiding in the development of effective control strategies and public health interventions.

Second, the Laplacian operator plays a key role in capturing spatial diffusion dynamics, representing the movement of individuals across a defined spatial domain. By accounting for concentration gradients, the model can capture variations in the density of infected individuals in different regions. 
This nuanced approach facilitates both local and global stability analyses, providing insight into equilibrium characteristics and the likelihood of epidemic persistence or extinction in specific spatial domains. 
The inclusion of this operator contributes to informed decision-making in public health planning and response, allowing for the identification of high-risk areas and the optimization of resource allocation.

In addition, the combined use of the $\CFC$ fractional time derivative and the Laplacian operator in the SEIR epidemic model can provide insight into the long-term behavior and stability of the system. 
Fractional calculus provides a broader perspective on the dynamics by allowing the analysis of differential equations of fractional order. 
This can help to reveal new phenomena and uncover additional aspects of epidemic spread that are not captured by traditional integer-order models. 
The inclusion of the Laplacian operator further enhances the model's ability to capture complex interactions, potentially leading to more accurate predictions and control strategies for epidemic outbreaks.
These modifications allow for the consideration of memory effects and provide a deeper understanding of the long-term behavior of epidemics. 
By using these advanced mathematical tools, researchers can improve the accuracy and realism of their models, ultimately contributing to better strategies for mitigating and controlling infectious diseases.

Nevertheless, challenges arise, including the determination of a solution to our system. 
We encounter significant hurdles in establishing the adjoint system associated with the proposed problem and subsequently deriving the essential optimality conditions.
In addition, complications arise in discretizing the problem to obtain numerical results.
Due to the diverse characteristics of these operators, many researchers and scientists have published relevant works in various fields \cite{Arshad2020, Hikal2021, SidiAmmi2022}.

In this article, we will study a new extension of the SEIR epidemic model, where we integrate the spatial behavior of the population and the control term representing the vaccination program, because it is sometimes considered as an effective means to prevent and control the spread of infection. 
Our goal is to minimize the number of infected people for the proposed fractional system coupled with no-flux boundary conditions by incorporating the $\CFC$ fractional time derivative and the Laplacian operator for the spread of the disease.

This paper is organized as follows. 
In Section~\ref{S2}, several key and crucial definitions related to the $\CF$ fractional calculus and its properties are presented. 
In Section~\ref{S3}, the SEIR fractional optimal control model is presented. 
While in Sections~\ref{S4} and \ref{S5}, we prove the existence, uniqueness, and positivity of the solution and the existence of an optimal solution to the proposed model. 
Furthermore, section~\ref{S6} is devoted to the determination of the necessary optimal conditions. 
Before concluding the present study, interesting numerical approximations that illustrate the relationship between the spread of the disease and changes in the order of the derivative are explained in Section~\ref{S7}. 
Finally, the conclusions of the study are discussed in detail in Section~\ref{S8}.

% Section 2
% ===============================================================
\section{Preliminary results}\label{S2}
In this section, we will recall some basic properties of the $\CFC$ fractional time derivative and the $\CF$ fractional integral. 
To do this, 
let $\alpha\in (0,1)$, $T>0$, $f\in H^1(0,T)$, and $t\in (0,T)$. 
We let $\gamma = \frac{\alpha}{1-\alpha}$ and $M(\alpha)$ be a normalization function such as $M(0)=M(1)=1$.

%%%%%%%%%%%%%%%%%%%%%%%%% Definition 1
\begin{definition}[{\cite[Page 2]{Abdeljawad2017}}]\label{D1}\ \\[-.5cm]
\begin{itemize}

\item[a.] The $\CFC$ fractional derivative of $f$ with base point $0$ of order $\alpha$ is defined at point $t$ by
\begin{equation}\label{E2.1}
   { }^{\CFC} \Dat f(t)
   =\frac{M(\alpha)}{1-\alpha} \int_0^t f^{\prime}(y) \,e^{-\gamma (t-y)} \,dy.
\end{equation}

\item[b.] The backward $\CFC$ fractional derivative with base point $T$, is defined by
\begin{equation}\label{E2.2}
{ }^{\CFC}_{T} \Dat f(t) = -\frac{M(\alpha)}{1-\alpha} \int_t^T f^{\prime}(y) \,e^{-\gamma (y-t)} \,dy.
\end{equation}
\end{itemize}
\end{definition}

%%%%%%%%%%%%%%%%%%%%%%%%% Remark 1
%\begin{remark}\label{R2}\label{R1}
% If we let $\alpha\rightarrow 1$ in \eqref{E2.1}, then we obtain the usual derivative $\partial_t$.
Note that if we let $\alpha\to1$ in \eqref{E2.1}, then we get the usual derivative $\partial_t$.
%\end{remark}

%%%%%%%%%%%%%%%%%%%%%%%%% Definition 2
\begin{definition}[{\cite[Page 1]{Abdeljawad2017}}]\label{D2}
The $\CF$ fractional integral operator with base point $0$, is written as
\begin{equation}\label{E2.3}
{ }^{\CF} I^{\alpha} f(t) 
= \frac{1-\alpha}{M(\alpha)} f(t) + \frac{\alpha}{M(\alpha)}
\int_0^t f(s)\,ds.
\end{equation}
\end{definition}
The following Lemma~\ref{L1} and Lemma~\ref{L2} (or Corollary~\ref{C1}) will help us to prove the existence of a positive unique solution to our fractional model.

%%%%%%%%%%%%%%%%%%%%%%%%% Lemma 1
\begin{lemma}[{\cite[Page 2]{Abdeljawad2017}}]\label{L1}
With the previous assumptions, we have 
\begin{equation}\label{E2.4}
  { }^{\CF} I^{\alpha} \bigl( { }^{\CFC} 
   \Dat f(t)\bigr) = f(t) - f(0).
\end{equation}
\end{lemma}

%%%%%%%%%%%%%%%%%%%%%%%%% Lemma 2
\begin{lemma}\label{L2}
Let $\varphi$ a continuous function on $[0, T]$. Then
\begin{equation}\label{E2.5}
    { }^{\CFC} \Dat \varphi(t) \cdot \varphi(t) 
    \ge \frac{1}{2} { }^{\CFC} \Dat \bigl(\varphi^2(t)\bigr).
\end{equation}
\end{lemma}

%%%%%%%%%%%%%%%%%%%%%%%%% Proof Lemma 2
\begin{proof}
Let us rewrite the inequality~\eqref{E2.5} in the form
\begin{equation*}
\begin{split}
{ }^{\CFC} \Dat \varphi(t) \cdot \varphi(t) - \frac{1}{2} { }^{\CFC} 
\Dat \bigl(\varphi^2(t)\bigr) 
&=  \frac{M(\alpha)}{1-\alpha} \Big(\varphi(t) 
   \int_0^t \varphi^{\prime}(y)\, e^{-\gamma (t-y)} \,dy\\
&\qquad -\frac{1}{2} \int_0^t 2\varphi(y)\varphi^{\prime}(y) \,e^{-\gamma (t-y)} \,dy\Big)\\ 
&= \frac{M(\alpha)}{1-\alpha} \int_0^t \bigl(\varphi(t) - \varphi(y)\bigr)\varphi^{\prime}(y) \,e^{-\gamma (t-y)} \,dy\\ 
&= \frac{M(\alpha)}{1-\alpha} \int_0^t \Bigl(\int_y^t \varphi^{\prime}(s)\,ds\Bigr)\varphi^{\prime}(y) \,e^{-\gamma (t-y)} \,dy\\  
&= \frac{M(\alpha)}{1-\alpha} \int_0^t \varphi^{\prime}(s) \Bigl( \int_0^s \varphi^{\prime}(y) \,e^{-\gamma (t-y)} \,dy\Bigr) \,ds.
\end{split}
\end{equation*}
Since
\begin{equation*}
\begin{split}
I &:= \frac{M(\alpha)}{1-\alpha} \int_0^t \varphi^{\prime}(s) 
  \Bigl( \int_0^s \varphi^{\prime}(y) \,e^{-\gamma (t-y)} \,dy\Bigr) \,ds\\ 
&= \frac{M(\alpha)}{2(1-\alpha)} \int_0^t e^{\gamma (t-s)} \frac{\partial}{\partial s}\biggl( \Bigl( \int_0^s \varphi^{\prime}(y) \,e^{-\gamma (t-y)} \,d y\Bigr)^2 \biggr)\,ds\\ 
&= \frac{M(\alpha)}{2(1-\alpha)} 
\biggl[ e^{\gamma (t-s)} \Bigl( 
\int_0^s \varphi^{\prime}(y) \,e^{-\gamma (t-y)} \,dy\Bigr)^2 \biggr]_{s=0}^{s=t}\\
&\qquad- \frac{M(\alpha)}{2(1-\alpha)}\int_0^t 
\Bigl(-\gamma e^{\gamma (t-s)}\Bigr) \Bigl( \int_0^s \varphi^{\prime}(y) \,e^{-\gamma (t-y)} \,dy\Bigr)^2 \,ds\\
&= \frac{M(\alpha)}{2(1-\alpha)} 
\Bigl( \int_0^t \varphi^{\prime}(y) \,e^{-\gamma (t-y)} \,dy\Bigr)^2 + \frac{M(\alpha)}{2(1-\alpha)} \gamma 
 \int_0^t e^{\gamma (t-s)} \Bigl( \int_0^s \varphi^{\prime}(y) \,e^{\gamma (t-y)} d y\Bigr)^2 \,ds \ge 0.
\end{split}
\end{equation*}
At this point, the proof is complete.
\end{proof}
Under the assumptions of Lemma~\ref{L2}, we have the following result
%%%%%%%%%%%%%%%%%%%%%%%%% Corollary 1
\begin{corollary}\label{C1}
Let $\varphi\colon[0, T] \to L^2(\Omega)$. 
Assume that 
there exists the $\CFC$ fractional derivative of $\varphi$.
Then,
\begin{equation}\label{E2.6}
\bigl({ }^{\CFC} \Dat \varphi(t), \varphi(t)\bigr)_{L^2(\Omega)} 
\ge \frac{1}{2} { }^{\CFC} \Dat \|\varphi(t)\|_{L^2(\Omega)}^2.
\end{equation}
\end{corollary}

The next section introduces the time-fractional SEIR PDE model under consideration.
% ====================================================================
\section{Mathematical model}\label{S3}
In the epidemiological literature, treatment and vaccination strategies are used to reduce the spread of infectious diseases or to achieve durable immunity in the population by analyzing the consequences of vaccinating vulnerable individuals and treating infected individuals. 
We assume that the total population $N$ consists of four subgroups of individuals:
\begin{description}\setlength{\itemsep}{0cm}
    \item[Susceptible $S(t,x)$:] Individuals in this compartment are susceptible to the disease but have not yet been infected.
    \item[Exposed $E(t,x)$:] Individuals in this compartment have been exposed to the infectious agent, but are not yet infectious themselves. This latent period represents the time between exposure to the pathogen and the onset of infectivity.
    \item[Infectious $I(t,x)$:] Individuals in this compartment are infectious and can transmit the disease to susceptible individuals.
    \item[Removed $R(t,x)$:] Individuals in this compartment have recovered from infection and are considered to have acquired immunity.
\end{description}
 We assume that the vaccines of all susceptible individuals are transferred directly to the category of removed individuals. 
Table~\ref{Tab1} presents the transmission coefficients applicable to the SEIR model, while Figure~\ref{F1} provides a comprehensive visualization of the transmission dynamics among the four categories.

\begin{table}[ht]
%%%%%%%%%%%%%%%%%%%%%%%%% Table 1
\begin{minipage}{0.49\linewidth}
\centering
\caption{Transmission coefficients for the SEIR model.}\label{Tab1}
\begin{tabular}{|c||c|}
\hline 
$1>\beta>0$ & Birth rate\\ 
\hline
$1>\kappa>0$ & Disease transmission rate\\ 
\hline
$1>\mu>0$ & Effective contact rate\\ 
\hline
$1>\xi>0$ & Natural mortality rate\\ 
\hline
$1>\eta>0$ & Recovery rate\\ 
\hline
\end{tabular}
\end{minipage}
\hfill
%%%%%%%%%%%%%%%%%%%%%%%%% Tikz
\begin{minipage}{0.49\linewidth}
\centering
\begin{tikzpicture}[node distance=2cm]
\node (S) [rectangle, draw, minimum size=0.7cm, fill=red!30] {S};
\node (E) [rectangle, draw, minimum size=0.7cm, fill=orange!30, right of=S] {E};
\node (I) [rectangle, draw, minimum size=0.7cm, fill=blue!30, right of=E] {I};
\node (R) [rectangle, draw, minimum size=0.7cm, fill=green!30, right of=I] {R}; 
% Arrows with text label
\draw[->] (-1.5,0) -- ++(S) node[midway,above]{$\beta N$};
\draw[->] (S.north) -| (0,1) node[near end,left]{$\xi S$};
\draw [->] (S) -- (E) node[midway,above]{$\mu SI$};
\draw[->] (E.north) -| (2,1) node[near end,left]{$\xi E$};
\draw [->] (E) -- (I) node[midway,above]{$\kappa E$};
\draw[->] (I.north) -| (4,1) node[near end,left]{$\xi I$};
\draw [->] (I) -- (R) node[midway,above]{$\eta I$};
\draw[->] (R.north) -| (6,1) node[near end,left]{$\xi R$};
\draw [->] (S.south) -| (0,-1) -- (6,-1) node[midway,above]{$u S$} -| (R.south) ;
\end{tikzpicture}

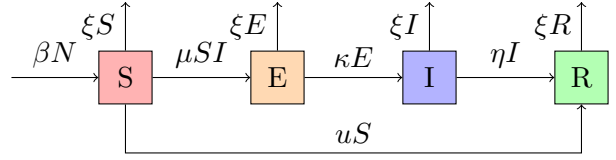
\captionof{figure}{Transmission dynamics in the SEIR model.}\label{F1}
\end{minipage}
\end{table}

Let $\Omega$ be a fixed and bounded domain in $\mathbb{R}^2$ with a smooth boundary denoted by $\partial \Omega$. 
To describe the spatial spreading effect of the disease, we assume that $\lambda_1,\lambda_2,\lambda_3,\lambda_4 >0$ are the respective diffusion coefficients for the four compartments. 
The optimal control system is then defined by
% \triangleq replaced by =:
\begin{equation}\label{E3.1}
\begin{cases}
{ }^{\CFC} \Dat S - \lambda_1 \Delta S 
&= \beta N - \mu S I - (\xi + u) S =: \Phi_1(S, E, I, R), \\
{ }^{\CFC} \Dat E - \lambda_2 \Delta E 
&= \mu S I -(\xi+ \kappa) E =: \Phi_2(S, E, I, R),\\
{ }^{\CFC} \Dat I - \lambda_3 \Delta I 
&= \kappa E - (\xi+ \eta) I =: \Phi_3(S, E, I, R),\\
{ }^{\CFC} \Dat R - \lambda_4 \Delta R 
&= uS + \eta I -\xi R =: \Phi_4(S, E, I, R),
\end{cases} (t, x) \in \Omega_T = [0, T]\times\Omega, 
\end{equation}
with
\begin{equation}\label{E3.2}
   \nabla S\cdot\vec{n} = \nabla E\cdot\vec{n} = \nabla I\cdot\vec{n} = \nabla R\cdot\vec{n} = 0, 
   \quad (t,x)\in \partial\Omega_T = [0, T]\times\partial\Omega,
\end{equation}
where $\vec{n}$ being the normal to $\partial\Omega$, and
\begin{equation}\label{E3.3}
     S(0,x)=S_0, \quad E(0,x)=E_0, \quad I(0,x)=I_0, \quad R(0,x)=R_0, \quad  x \in \Omega.
\end{equation}
The objective functional that will contribute to reducing the intensity of infected people and the costs of the vaccination program, is given by
\begin{equation}\label{E3.4}
\cJ\bigl( (S, E, I, R), u\bigr) 
= \int_\Omega I^2(T, x) \,dx 
+ \int_0^T \int_\Omega I^2(t, x) \,dx dt 
+ \sigma \int_0^T \int_\Omega u^2(t, x) \,dx dt,
\end{equation}
where $\sigma$ is a constant weight associated with the vaccination control $u$, and
\begin{equation}\label{E3.5}
    u\in U_{ad} = \bigl\{w \in L^{\infty}(\Omega_T) / \|w\|_{L^{\infty}(\Omega_T)}<1 \text { and } w>0\bigr\}.
\end{equation}
%%% 'ad' stands for 'adjoint'

Let $\omega = (\omega_1, \omega_2, \omega_3, \omega_4) = (S, E, I, R)$, 
$\omega^0 =( \omega^0_1, \omega^0_2, \omega^0_3, \omega^0_4)=(S_0, E_0, I_0, R_0)$, 
and $\lambda = (\lambda_1, \lambda_2, \lambda_3, \lambda_4)$. 
Then, the problem~\eqref{E3.1}--\eqref{E3.3} can be rewritten in the form
\begin{equation}\label{E3.6}
\begin{cases}
{ }^{\CFC} \Dat \omega(t) + \cL \omega(t) = \Phi(\omega(t)),\\
\omega(0)=\omega^0,
\end{cases} t\in[0,T],
\end{equation}
with $\omega(t)(\cdot)=\omega(t,\cdot)$, $\Phi = (\Phi_1, \Phi_2, \Phi_3, \Phi_4)$, and
\begin{equation*}
\begin{gathered}
\cL\colon\quad \begin{array}{l}
D(\cL)\subset L(\Omega)\to L(\Omega) \\
y \quad \to \quad -\lambda\Delta y
\end{array},\\
  D(\cL)=\Big\{\vartheta\in \bigl(H^2(\Omega)\bigr)^4 / \ \nabla \vartheta_i\cdot\vec{n} = 0, i=1,2,3,4\Big\},
\end{gathered}
\end{equation*}
where $L(\Omega) = \bigl(L^2(\Omega)\bigr)^4$.
\footnotetext[1]{Recall that $H^2(\Omega)$ is the Sobolev space $W^{2,2}(\Omega)$ (it is a Hilbert space).
$H^1_0(\Omega)$ is the closure of the smooth functions with compact support in $\Omega$ in $H^1(\Omega) = W^{1, 2 }(\Omega)$ under the associated Sobolev norm.}

% ====================================================================
\section{Existence and uniqueness of the solution}\label{S4}
We multiply the problem \eqref{E3.6} by a function $\varphi \in \cH(\Omega) = \bigl(H^1_0(\Omega)\bigr)^4$ and get 
\begin{equation}\label{E4.1}
\begin{cases}
\bigl\langle { }^{\CFC} \Dat\omega , \varphi \bigr\rangle 
+ \cB(\omega, \varphi) 
= \bigl\langle \Phi(\omega) , \varphi \bigr\rangle,\\
\bigl\langle \omega(0), \varphi \bigl\rangle 
= \bigl\langle \omega^0, \varphi \bigr\rangle,
\end{cases} t\in[0,T],
\end{equation}
where $\langle \cdot , \cdot \rangle := \langle \cdot , \cdot \rangle_{L(\Omega)}$ and $\cB$ is the bilinear form defined in $\cH(\Omega)$ by
\begin{equation*}
      \cB(\omega, \varphi) = \langle \nabla\omega , \nabla \varphi \rangle = \int_\Omega \nabla\omega \nabla \varphi \,dx =: \langle \cL\omega, \varphi \rangle.
\end{equation*}
Assuming that $-\Delta$ is a uniformly elliptic operator, the spectrum of $\cL$ consists of eigenvalues $\{\varrho_k\}_{k=1}^\infty$ and their corresponding orthogonal eigenfunctions $\{\omega^k\}_{k=1}^\infty$ within $D(\cL)$, satisfying  
% the condition that
\begin{equation*}
    \cB(\omega^k, \varphi) = \varrho_k \langle \omega^k , \varphi \rangle.
\end{equation*}
Hence, we have
\begin{equation}\label{E4.2}
    \|\varphi\|_{\cH(\Omega)}^2 
    = \sum_{k=0}^\infty \varrho_k \langle \varphi, \omega^k \rangle.
\end{equation}

%%%%%%%%%%%%%%%%%%%%%%%%% Lemma 3
\begin{lemma}\label{L3}
Let $\omega^k \in H^1\bigl(0, T; L(\Omega)\bigr)$. 
Then, the solution of the problem
\begin{equation}\label{E4.3}
\begin{cases}
{ }^{\CFC} \Dat \omega^k + \varrho_k \omega^k 
   = \Phi^k(t) =:\Phi(\omega^k),\\
\omega^k(0)=\omega^k_0,
\end{cases} t\in[0,T],
\end{equation}
is given by
\begin{equation}\label{E4.4}
   \omega^k(t)= \zeta_k \exp(-\varpi_k t) \omega^k_0
   +\frac{(1-\alpha) \zeta_k}{M(\alpha)}\,\Phi^k(t) 
   + \Lambda_k \int_0^t \exp\bigl(-\varpi_k(t-s)\bigr) \,\Phi^k(s) \,ds,
\end{equation}
where
\begin{equation}\label{E4.5}
   \varpi_k = \frac{\alpha \varrho_k}{M(\alpha)+(1-\alpha) \varrho_k}, \quad
   \zeta_k
   = \frac{M(\alpha)}{M(\alpha)+(1-\alpha) \varrho_k}, \ \text{ and } \ \Lambda_k=\frac{\zeta_k \bigl(\alpha + (1-\alpha) \varpi_k\bigr)}{M(\alpha)}.
\end{equation}
\end{lemma}

%%%%%%%%%%%%%%%%%%%%%%%%% Proof Lemma 3
\begin{proof}
Applying the $\CF$ fractional integral to the problem~\eqref{E4.3}, we obtain
\begin{equation}\label{E4.6}
   \omega^k(t) - \omega^k(0) 
   = -\frac{(1-\alpha)\varrho_k}{M(\alpha)}\,\omega^k(t) 
   - \frac{\alpha \varrho_k}{M(\alpha)} \int_0^t \omega^k(s)\,ds 
   + \frac{1-\alpha}{M(\alpha)} \,\Phi^k(t) + \frac{\alpha}{M(\alpha)} \int_0^t \Phi^k(s)\,ds.
\end{equation}
If we subject both sides of the equation~\eqref{E4.6} to the Laplace transform, we get
\begin{equation*}
   \hat{\omega}^k(p) - \frac{\omega^k(0)}{p} 
   = -\frac{(1-\alpha)\varrho_k}{M(\alpha)}\,\hat{\omega}^k(p) 
   - \frac{\alpha \varrho_k}{M(\alpha)} \frac{\hat{\omega}^k(p)}{p} 
   + \frac{1-\alpha}{M(\alpha)} \,\hat{\Phi}^k(p)
   + \frac{\alpha}{M(\alpha)} \frac{\hat{\Phi}^k(p)}{p}.
\end{equation*}
Then,
\begin{equation*}
  \frac{pM(\alpha)+ p(1-\alpha)\varrho_k + \alpha\varrho_k}{pM(\alpha)}\, \hat{\omega}^k(p) 
  = \frac{\omega^k(0)}{p} + \frac{1-\alpha}{M(\alpha)}\,\hat{\Phi}^k(p)+ \frac{\alpha}{M(\alpha)} \frac{\hat{\Phi}^k(p)}{p}.
\end{equation*}
Now a straightforward calculation gives
\begin{equation*}
\begin{split}
\hat{\omega}^k(p) =& \frac{M(\alpha)}{M(\alpha)+ (1-\alpha)\varrho_k} 
\biggl(\frac{1}{p + \frac{\alpha\varrho_k}{M(\alpha)+ (1-\alpha)\varrho_k}}\biggr) \omega^k(0)\\
&\qquad + \frac{1-\alpha}{M(\alpha)+ (1-\alpha)\varrho_k}  
 \biggl(\frac{p}{p + \frac{\alpha\varrho_k}{M(\alpha)+ (1-\alpha)\varrho_k}}\biggr) \hat{\Phi}^k(p)\\
&\qquad+ \frac{\alpha}{M(\alpha)+ (1-\alpha)\varrho_k}  
  \biggl(\frac{1}{p + \frac{\alpha\varrho_k}{M(\alpha)+ (1-\alpha)\varrho_k}}\biggr) \hat{\Phi}^k(p).
\end{split}
\end{equation*}
We set
\begin{equation*}
    \varpi_k 
    = \frac{\alpha \varrho_k}{M(\alpha)+(1-\alpha) \varrho_k} \ \text{ and } \ \zeta_k = \frac{M(\alpha)}{M(\alpha)+(1-\alpha) \varrho_k}.
\end{equation*}
Therefore, we obtain
\begin{equation*}
\hat{\omega}^k(p) = \zeta_k \frac{1}{p + \varpi_k} \,\omega^k(0) 
+ \frac{(1-\alpha)\zeta_k}{M(\alpha)} \frac{1}{1 + (\varpi_k^{-1}p)^{-1}} \,\hat{\Phi}^k(p)+ \frac{\alpha \zeta_k}{M(\alpha)} \frac{1}{p + \varpi_k} \,\hat{\Phi}^k(p).
\end{equation*}
Recall that
\begin{equation*}
  % \mathscr{L}^{-1}(F(p)G(p))(t) = (f * g)(t).
   \mathscr{L}^{-1}\bigl(\hat{f}(p) \hat{g}(p)\bigr)(t) = (f * g)(t), 
\end{equation*}
where '$*$' denotes a convolution operator. 
Hence by applying the inverse Laplace transform, the solution of \eqref{E4.3} is given by
\begin{equation*}
    \omega^k(t)= \zeta_k \exp(-\varpi_k t) \omega^k_0
    +\frac{(1-\alpha) \zeta_k}{M(\alpha)} \,\Phi^k(t) 
    + \Lambda_k \int_0^t \exp\bigl[-\varpi_k(t-s)\bigr] \Phi^k(s) \,ds,
\end{equation*}
where $\Lambda_k=\frac{\zeta_k (\alpha + (1-\alpha) \varpi_k)}{M(\alpha)}$.
\end{proof}

\medskip
Building on the result of Lemma~\ref{L3}, we establish the following important result
%%%%%%%%%%%%%%%%%%%%%%%%% Theorem 1
\begin{theorem}\label{T1}
Let $\omega^0 \in L(\Omega)$. 
Then, \eqref{E4.1} has one and only one solution in $L^2\bigl(0, T; \cH(\Omega)\bigr) \cap L^\infty\bigl(0, T; L(\Omega)\bigr)$ given by
\begin{equation*}
     \omega(t, x) = \sum_{k=1}^{\infty} 
     \Bigl[\zeta_k \exp\bigl(-\varpi_k t\bigr) \,\omega^k_0
     +\frac{(1-\alpha) \zeta_k}{M(\alpha)}\,\Phi^k(t)\Bigr] e_k 
     + \Bigl[\Lambda_k \int_0^t \exp\bigl(-\varpi_k(t-s)\bigr)\, \Phi^k(s) \,ds\Bigr] e_k,
\end{equation*}
where the constants $\zeta_k, \varpi_k$, and $\Lambda_k$ are given by \eqref{E4.5}. 
In addition, there is a constant $C:= C(\alpha, T, \varrho_1)$ with
\begin{equation*}
   \|\omega\|_{L^2(0, T; \cH(\Omega))} 
     \le C \bigl(\|\omega^0\|_{\cH(\Omega)} + \|\Phi\|_{L(\Omega_T)} \bigr),
\end{equation*}
 i.e.\ the solution can be bounded by the data $\omega^0$, $\Phi$.
\end{theorem}
%%%%%%%%%%%%%%%%%%%%%%%%% Proof Theorem 1
\begin{proof}
Let $\mathcal{V}_m$ be a subspace of $\cH(\Omega)$ generated by $\{e_i\}_{i=1}^m$. We want to find the $\omega^m$ solution of the following FDE system
\begin{equation*}
\begin{cases}
\bigl\langle { }^{\CFC} \Dat\omega^m , \varphi \bigr\rangle 
+ \cB(\omega^m, \varphi) = \bigl\langle \Phi(\omega^m) , \varphi \bigr\rangle,\\
\omega(0) = \omega^0.
\end{cases} \forall \varphi \in \mathcal{V}_m, \ t\in[0,T],
\end{equation*}
Using the fact that $\cB(\omega, e_k) = \varrho_k \omega^k$, we obtain
\begin{equation*}
\begin{cases}
{ }^{\CFC} \Dat \omega^k + \varrho_k \omega^k 
    = \Phi^k(t) =: \Phi(\omega^k),\\
\omega^k(0)=\omega^k_0,
\end{cases} t\in[0,T],
\end{equation*}
which admits a solution given by \eqref{E4.4}.
Now we have to prove that the solution of \eqref{E4.1} is unique and belongs to $L^2\bigl(0, T; \cH(\Omega)\bigr) \cap L^\infty\bigl(0, T; L(\Omega)\bigr)$.

Since $\omega^m \in \mathcal{V}_m$, we have
\begin{equation*}
   \omega^m = \sum_{k=1}^m \langle \omega , e_k \rangle e_k 
   = \sum_{k=1}^m \omega^k e_k.
\end{equation*}
Moreover,
\begin{equation*}
   \omega^m(t, x) = \sum_{k=1}^{m} \Bigl[\zeta_k \exp\bigl(-\varpi_k t\bigr) \omega^k_0+\frac{(1-\alpha) \zeta_k}{M(\alpha)} \Phi^k(t)\Bigr] e_k 
   + \sum_{k=1}^{m} \Bigl[\Lambda_k \int_0^t \exp\bigl(-\varpi_k(t-s)\bigr) \Phi^k(s) \,ds\Bigr] e_k.
\end{equation*}
Let $m, p \in \mathbb{N}^*$ such that $p>m$. Then,
\begin{equation*}
     \omega^p(t, x) - \omega^m(t, x) = \sum_{k=m+1}^p \omega^k e_k.
\end{equation*}
We also have 
\begin{equation*}
\begin{split}
    \cB(\omega^p - \omega^m, \omega^p - \omega^m) 
    &= \sum_{k=m+1}^p \varrho_k (\omega^k)^2\\
    &\le 3 \sum_{k=m+1}^{p} \varrho_k\Bigl[\zeta_k \exp(-\varpi_k t) \omega^k_0\Bigr]^2 
    + 3 \sum_{k=m+1}^{p} \varrho_k\Bigl[\frac{(1-\alpha) \zeta_k}{M(\alpha)}\,\Phi^k(t)\Bigr]^2\\
    &\qquad + 3 \sum_{k=m+1}^{p} \varrho_k \Bigl[\Lambda_k \int_0^t \exp\bigl(-\varpi_k(t-s)\bigr) \,\Phi^k(s) \,ds\Bigr]^2.
\end{split}
\end{equation*}
Because of
\begin{align*}
\varpi_k &= \frac{\alpha \varrho_k}{M(\alpha)+(1-\alpha) \varrho_k} 
\le \frac{\alpha \varrho_k}{(1-\alpha) \varrho_k} 
= \frac{\alpha}{(1-\alpha)},\\
\varrho_k \zeta_k^2 &= \frac{\varrho_k M(\alpha)^2}{\bigl(M(\alpha)+(1-\alpha) \varrho_k\bigr)^2} 
\le \frac{M(\alpha)^2}{(1-\alpha)^2 \varrho_k} 
\le \frac{M(\alpha)^2}{(1-\alpha)^2 \varrho_1}, \\
\varrho_k \Lambda_k^2 &= \varrho_k \zeta_k^2 \frac{ \bigl(\alpha + (1-\alpha) \varpi_k\bigr)^2}{M(\alpha)^2} 
\le \frac{M(\alpha)^2}{(1-\alpha)^2 \varrho_1} \frac{ \bigl(\alpha + \frac{\alpha (1-\alpha)}{(1-\alpha)}\bigr)^2}{M(\alpha)^2} 
= \frac{4\alpha^2}{(1-\alpha)^2 \varrho_1},
\end{align*}
we obtain
\begin{equation*}
\|\omega^p - \omega^m\|_{L^2\bigl(0, T; \cH(\Omega)\bigr)}^2 
= \int_0^T \cB(\omega^p - \omega^m, \omega^p - \omega^m) \,dt 
\le 3 ( \mathcal{A}_1 + \mathcal{A}_2 + \mathcal{A}_3 ),
\end{equation*}
where
\begin{equation*}
\begin{split}
\bullet \quad \mathcal{A}_1 :&= \sum_{k=m+1}^{p} \varrho_k \int_0^T \Bigl[\zeta_k \exp(-\varpi_k t) \omega^k_0\Bigr]^2 dt 
= \sum_{k=m+1}^{p} \varrho_k \zeta_k^2 \bigl[\omega^k_0\bigr]^2 \int_0^T \bigl[\exp(-\varpi_k t)\bigr]^2\,dt\\ 
&\le \sum_{k=m+1}^{p} \frac{M(\alpha)^2}{(1-\alpha)^2 \varrho_1} \bigl[\omega^k_0\bigr]^2 T 
= \frac{M(\alpha)^2 T}{(1-\alpha)^2 \varrho_1} \sum_{k=m+1}^{p} \bigl[\omega^k_0\bigr]^2.\\
%%%% A2
\bullet \quad \mathcal{A}_2 :&= \sum_{k=m+1}^{p} \varrho_k \int_0^T \biggl[\frac{(1-\alpha) \zeta_k}{M(\alpha)} \,\Phi^k(t)\biggr]^2\,dt 
= \sum_{k=m+1}^{p} \varrho_k \zeta_k^2 \frac{(1-\alpha)^2}{M(\alpha)^2} \int_0^T \bigl[\Phi^k(t)\bigr]^2 \,dt\\
&\le \sum_{k=m+1}^{p} \frac{M(\alpha)^2}{(1-\alpha)^2 \varrho_1} \frac{(1-\alpha)^2}{M(\alpha)^2} \int_0^T \bigl[\Phi^k(t)\bigr]^2\,dt 
= \frac{1}{\varrho_1} \sum_{k=m+1}^{p} \int_0^T \bigl[\Phi^k(t)\bigr]^2\,dt.\\
%%%%%%% A3
\bullet \quad \mathcal{A}_3 :&= \sum_{k=m+1}^{p} \varrho_k \int_0^T \biggl[\Lambda_k \int_0^t \exp\bigl(-\varpi_k(t-s)\bigr)\, \Phi^k(s) \,ds\biggr]^2 \,dt\\
&= \sum_{k=m+1}^{p} \varrho_k \Lambda_k^2 \int_0^T \biggl[ \int_0^t 
\exp\bigl(-\varpi_k(t-s)\bigr) \,\Phi^k(s) \,ds\biggr]^2 \,dt\\
&\le \sum_{k=m+1}^{p} \frac{4\alpha^2}{(1-\alpha)^2 \varrho_1} 
\int_0^T \biggl[ \int_0^t \exp\bigl(-\varpi_k(t-s)\bigr)\,\Phi^k(s) \,ds\biggr]^2\, dt\\
%&\overset{C.S\footnote{Cauchy–Schwarz inequality.}}{\leq} 
%
&\le \frac{4\alpha^2}{(1-\alpha)^2 \varrho_1} \sum_{k=m+1}^{p}  
\int_0^T \biggl[ \Bigl(\int_0^t \bigl(\exp(-\varpi_kz)\bigr)\,dz\Bigr) \Bigl(\int_0^t \bigl(\Phi^k(s)\bigr)^2\,ds\Bigr)\biggr] \,dt\\
&\le \frac{4\alpha^2 T^2}{(1-\alpha)^2 \varrho_1} \sum_{k=m+1}^{p} 
\int_0^T \bigl[\Phi^k(t)\bigr]^2 \,dt.\\
\end{split}
% \footnotetext[5]{C.S: Cauchy–Schwarz inequality.}
\end{equation*}
Afterwards,
\begin{equation}\label{E4.7}
\begin{split}
\|\omega^p - \omega^m\|_{L^2(0, T; \cH(\Omega))} 
&\le C_1(\alpha, T, \varrho_1) \biggl( \sum_{k=m+1}^{p} \bigl[\omega^k_0\bigr]^2\biggr)^\frac{1}{2}\\ 
&\qquad + C_2(\alpha, T, \varrho_1) \biggl( \sum_{k=m+1}^{p}  
\int_0^T \bigl[\Phi^k(t)\bigr]^2 \,dt\biggr)^\frac{1}{2}.
\end{split}
\end{equation}
From the fact that $\omega^0 \in L(\Omega)$, we have 
\begin{equation*}
    \lim_{p, m \to \infty} \|\omega^p - \omega^m\|_{L^2(0, T; \cH(\Omega))} = 0,
\end{equation*}
which implies that $(\omega^m)$ is a Cauchy sequence in $L^2\bigl(0, T; \cH(\Omega)\bigr)$. 
Therefore, we obtain
\begin{equation*}
     \omega^m\to\omega \ \text{ in } \ L^2\bigl(0, T; \cH(\Omega)\bigr).
\end{equation*}
In addition, by \eqref{E4.7}, for $C = \max\bigl(C_1(\alpha, T, \varrho_1), C_2(\alpha, T, \varrho_1) \bigr)$, we get
\begin{equation*}
    \|\omega\|_{L^2(0, T; \cH(\Omega))} 
    \le C \bigl(\|\omega^0\|_{\cH(\Omega)} + \|\Phi\|_{L(\Omega_T)} \bigr).
\end{equation*}

%%%%%%%%%%%%
% In the following sections of this article 
In the sequel we will assume that the birth rate $\beta$ is either less than or equal to the natural mortality rate $\xi$.

Now, starting from \eqref{E3.1}, we have
\begin{equation*}
    { }^{\CFC} \Dat N - \lambda_1\Delta \omega_1 - \lambda_2\Delta \omega_2 - \lambda_3\Delta \omega_3 - \lambda_4\Delta \omega_4 = \beta N - \xi N,
\end{equation*}
where $N = \omega_1 + \omega_2 + \omega_3 + \omega_4$. 
Due to the linearity of $-\Delta$, there exists a constant $C$ such that
\begin{equation*}
    { }^{\CFC} \Dat N - C\Delta N \le (\beta - \xi) N.
\end{equation*}
Utilizing \eqref{E2.4} and \eqref{E2.6}, we get
\begin{equation*}
\begin{split}
\|N\|_{L^2(\Omega)}^2 
&\le \|N^0\|_{L^2(\Omega)}^2 + \frac{1-\alpha}{M(\alpha)} (\beta - \xi) \|N\|_{L^2(\Omega)}^2 + \frac{\alpha}{M(\alpha)} (\beta - \xi) \int_0^t \|N(s)\|_{L^2(\Omega)}^2 \,ds\\
&\le \|N^0\|_{L^2(\Omega)}^2 + \frac{\alpha(\beta - \xi)}{M(\alpha)} \int_0^t \|N(s)\|_{L^2(\Omega)}^2 \,ds.
\end{split}
\end{equation*}
By Gronwall inequality, we obtain
\begin{equation*}
    \|N(t)\|_{L^2(\Omega)}^2 \leq \|N^0\|_{L^2(\Omega)}^2 \exp\Bigl[ \frac{\alpha(\beta - \xi)}{M(\alpha)} t\Bigr].
\end{equation*}
That is to say
\begin{equation*}
     \omega \in L^\infty\bigl(0, T; L(\Omega)\bigr).
\end{equation*}
\end{proof}

\medskip
The non-negativity of the unique solution is given by the following theorem.
%%%%%%%%%%%%%%%%%%%%%%%%% Theorem 2
\begin{theorem}\label{T2}
The solution of \eqref{E3.6} is positive if
\begin{equation*}
\theta := 1 - \frac{2(1-\alpha)}{M(\alpha)}\|N\|_{L^\infty(0, T; L^2(\Omega))} > 0.
\end{equation*}
\end{theorem}
%%%%%%%%%%%%%%%%%%%%%%%%% Proof Theorem 2
\begin{proof}
Let $\omega$ be the solution of \eqref{E3.6}. 
According to Theorem~\ref{T1}, $\omega\in L^\infty\bigl(0, T; L(\Omega)\bigr)$.

%%%%%%%%%%%%
First, let us show that $\omega_3 \ge0$. 
We introduce the partitioning $\omega_3 = \omega_3^+ + \omega_3^-$, where $\omega_3^+ = \max(\omega_3, 0)$ and $\omega_3^- = \max(-\omega_3, 0)$
and multiply the third equation of \eqref{E3.6} by $\omega_3^-$
% \footnote{$\omega_3 = \omega_3^+ + \omega_3^-$, where $\omega_3^+ = \max(\omega_3, 0)$ and $\omega_3^- = \max(-\omega_3, 0)$.},
and get
\begin{equation*}
   \bigl({ }^{\CFC} \Dat\omega_3, \omega_3^- \bigr)_{L^2(\Omega)} - \lambda_3\bigl(\Delta\omega_3, \omega_3^- \bigr)_{L^2(\Omega)} 
   = \kappa \bigl(\omega_2, \omega_3^- \bigr)_{L^2(\Omega)} 
   - (\xi + \eta)\bigl(\omega_3, \omega_3^- \bigr)_{L^2(\Omega)}.
\end{equation*}
According to the Cauchy–Schwarz inequality and \eqref{E2.6}, we have
\begin{equation*}
    \frac{1}{2}{ }^{\CFC} \Dat \|\omega_3^- \|_{L^2(\Omega)}^2 + \lambda_3\|\nabla\omega_3^- \|_{L^2(\Omega)}^2 
    \le\kappa \|\omega_2 \|_{L^2(\Omega)} \|\omega_3^- \|_{L^2(\Omega)} - (\xi + \eta)\|\omega_3^- \|_{L^2(\Omega)}^2.
\end{equation*}
Then,
\begin{equation*}
   \frac{1}{2}{ }^{\CFC} \Dat \|\omega_3^- \|_{L^2(\Omega)}^2 
   \le \kappa \|\omega_2 \|_{L^2(\Omega)} \|\omega_3^- \|_{L^2(\Omega)}.
\end{equation*}
By applying \eqref{E2.4}, we obtain
\begin{equation*}
\begin{split}
\|\omega_3^- \|_{L^2(\Omega)}^2 
&\le\|\omega_3^{0 -} \|_{L^2(\Omega)}^2 + 2\kappa\frac{1-\alpha}{M(\alpha)} \|\omega_2 \|_{L^2(\Omega)} \|\omega_3^- \|_{L^2(\Omega)} + \frac{2\kappa\alpha}{M(\alpha)} \int_0^t \|\omega_2(s) \|_{L^2(\Omega)} \|\omega_3^-(s) \|_{L^2(\Omega)} \,ds\\
&\le\|\omega_3^{0 -} \|_{L^2(\Omega)}^2 + \frac{2(1-\alpha)}{M(\alpha)} \|N\|_{L^2(\Omega)} \|\omega_3^- \|_{L^2(\Omega)} + \frac{2\kappa\alpha}{M(\alpha)} \int_0^t \|\omega_2(s) \|_{L^2(\Omega)} \|\omega_3^-(s) \|_{L^2(\Omega)} \,ds.
\end{split}
\end{equation*}
Since $\sqrt{a+b}\le\sqrt{a} + \sqrt{b}$, we have
\begin{equation*}
   \|\omega_3^- \|_{L^\infty(0, T; L^2(\Omega))} \leq \frac{1}{\theta}\|\omega_3^{0 -} \|_{L^2(\Omega)} + \frac{2\kappa\alpha}{M(\alpha)\theta} \int_0^t \|\omega_2(s)\|_{L^\infty(0, T; L^2(\Omega))} \|\omega_3^-(s)\|_{L^\infty(0, T; L^2(\Omega))} \,ds.
\end{equation*}
Applying the Gronwall inequality, we get
\begin{equation*}
  \|\omega_3^- \|_{L^\infty(0, T; L^2(\Omega))} 
  \le\frac{1}{\theta}\|\omega_3^{0 -}\|_{L^2(\Omega)} \exp\biggl[\frac{2\kappa\alpha}{M(\alpha)\theta}
  \int_0^t \|\omega_2(s)\|_{L^\infty(0, T; L^2(\Omega))} \,ds\biggr].
\end{equation*}
Afterwards, $\omega_3^- = 0$.

Next, we will show that $\omega_1 \ge0$. To do so, we multiply the first equation of \eqref{E3.6} by $\omega_1^-$ and obtain
\begin{equation*}
   \bigl({ }^{\CFC} \Dat \omega_1, \omega_1^- \bigr)_{L^2(\Omega)} - \lambda_1\bigl(\Delta\omega_1, \omega_1^- \bigr)_{L^2(\Omega)} 
= \beta \bigl(N, \omega_1^- \bigr)_{L^2(\Omega)} - \mu\bigl(\omega_1\omega_3, \omega_1^- \bigr)_{L^2(\Omega)} - (\xi + u)\bigl(\omega_1, \omega_1^- \bigr)_{L^2(\Omega)}.
\end{equation*}
Since $\omega_3 \ge0$, by the Cauchy–Schwarz inequality and \eqref{E2.6}, we get
\begin{equation*}
{ }^{\CFC} \Dat \|\omega_1^- \|_{L^2(\Omega)}^2 \leq 2\beta \|N\|_{L^2(\Omega)} \|\omega_1^- \|_{L^2(\Omega)}.
\end{equation*}
Applying \eqref{E2.4}, we obtain
\begin{equation*}
\|\omega_1^-\|_{L^\infty(0, T; L^2(\Omega))} 
\le\frac{1}{\theta}\|\omega_1^{0 -} \|_{L^2(\Omega)} + \frac{2\beta\alpha}{M(\alpha)\theta} 
\int_0^t \|N(s)\|_{L^\infty(0, T; L^2(\Omega))} \|\omega_1^-(s)\|_{L^\infty(0, T; L^2(\Omega))} \,ds.
\end{equation*}
According to the Gronwall inequality, we find
\begin{equation*}
    \|\omega_1^-\|_{L^\infty(0, T; L^2(\Omega))} 
    \le\frac{1}{\theta}\|\omega_1^{0 -}\|_{L^2(\Omega)}
    \exp\biggl[\frac{2\beta\alpha}{M(\alpha)\theta} \int_0^t \|N(s)\|_{L^\infty(0, T; L^2(\Omega))} \,ds\biggr].
\end{equation*}
Then, it follows $\omega_1^- = 0$.

We note that the same methodology gives $\omega_2 \ge0$ and $\omega_4 \ge0$.
\end{proof}

% ====================================================================
\section{Existence of an optimal solution}\label{S5}
To prove the existence of an optimal control, we use the technique of minimizing sequences.
The proof relies on several lemmas, which we introduce below.

%%%%%%%%%%%%%%%%%%%%%%%%% Lemma 4
\begin{lemma}[{\cite[Page 36]{Atanackovi2018}}]\label{L4}
Let $\nu, \varphi \in C^{\infty}(\Omega_T)$. Then
\begin{equation*}
\begin{split}
\int_0^T \bigl({ }^{\CFC} \Dat \nu\bigr)\,\varphi \,dt 
&= \int_0^T \nu \,\bigl({ }^{\CFC}_{T} \Dat \varphi\bigr)  \,dt\\
&\qquad+ \frac{1}{1-\alpha} \,\varphi(T, x) 
  \int_0^T \nu(t) \,e^{-\gamma (T-t)} \,dt 
   - \frac{1}{1-\alpha} \,\nu(0, x) \int_0^T \varphi(t) \,e^{-\gamma t} \,dt.
\end{split}
\end{equation*}
\end{lemma}

%%%%%%%%%%%%%%%%%%%%%%%%% Lemma 5
\begin{lemma}[{\cite[Page 7]{SidiAmmi2022}}]\label{L5}
Let $\tilde{y} \in L^{\infty}\bigl(0, T, L^2(\Omega)\bigr) \cap H^1\bigl(0, T, L^1(\Omega)\bigr)$, then there exists a positive
constant $k$ such that
\begin{equation*}
     \bigl\|\partial_t \tilde{y}\bigr\|_{L^1(0, T, L^1(\Omega))} 
      \le\frac{k}{E_{\alpha}(-\gamma T^{\alpha})}
          \bigl\|\tilde{y}\bigr\|_{L^{\infty}(0, T, L^2(\Omega))}.
\end{equation*}
\end{lemma}

\medskip
According to these two lemmas, we have the following theorem.
%%%%%%%%%%%%%%%%%%%%%%%%% Theorem 3
\begin{theorem}\label{T3}
The problem~\eqref{E3.1}--\eqref{E3.3} admits an optimal solution $\omega^*(u^*)\in L^\infty(\omega_T)$ which minimizes \eqref{E3.4}.
\end{theorem}
%%%%%%%%%%%%%%%%%%%%%%%%% Proof Theorem 3
\begin{proof}
Let $\bigl((\omega^n, u^n)\bigr)_n$ such as
\begin{equation*}
   \cJ(\omega^*,u^*) 
  = \lim_{n\to\infty} \cJ(\omega^n, u^n) 
  = \inf_{u\in U_{ad}}\bigl\{\cJ(\omega, u)\bigr\} ,
\end{equation*}
where $\omega^n = (\omega_i^n)_{i=1,2,3,4}$ and $u^n\in U_{ad}$.
Let $i\in\{1,2,3\}$, the couple $(\omega_i^n, u^n)$ satisfying the system
\begin{equation*}
\begin{cases}
{ }^{\CFC} \Dat\omega_i^n + \cL\omega_i^n = \Phi_i(\omega^n), &\text{ in } \Omega_T,\\ \vspace{0.2cm}
\nabla \omega_i^n\cdot\vec{n} = 0,  &\text{ on }  \partial\Omega_T,\\
\omega_i^n(0) = \omega_i^0,  &\text{ in }  \Omega.
\end{cases}
\end{equation*}
By the boundedness of $\omega_i^n$ ($|\omega_i^n| \le N$) and Theorem~\ref{T1}, the sequence $(\omega_i^n)$ is bounded in $L^\infty(0, T; L^2(\Omega))$ 
and in $L^2(0, T; \cH(\Omega))$. 
The second member $\Phi_i(\omega^n)$ is also bounded in $L^\infty(\Omega_T)$. 
So there is a posi\-tive constant $c$ such as
\begin{equation*}
  \bigl\| { }^{\CFC} \Dat\omega_i^n - \cL\omega_i^n \bigr\|_{L^2(\Omega_T)} \le c.
\end{equation*}
Then there is a subsequence of ($\omega^n$) denoted again by ($\omega^n$) such that
\begin{equation*}
\Bigg| \ \begin{aligned}
& { }^{\CFC} \Dat\omega_i^n - \cL\omega_i^n \rightharpoonup  \phi \text { weakly in } L^2(\Omega_T), \\
& \omega^n \rightharpoonup  \omega^* \text { weakly in } L^2\bigl(0, T ; \cH(\Omega)\bigr).
\end{aligned}
\end{equation*}
Set
\begin{equation*}
     \mathcal{K} = \Bigl\{w\in L^2\bigl(0, T ; \cH(\Omega)\bigr) / \partial_t w\in L^1\bigl(0, T; (L^1(\Omega))^4\bigr) \Bigr\}.
\end{equation*}
Since $\cH(\Omega)$ is compactly embedded in $L^2(\Omega)$, we conclude that $(\omega_i^n)$ is compact in $L(\Omega)$. 
By Lemma~\ref{L5} we get that $(\partial_t \omega_i^n)$ is bounded in $L^1(0, T; L^1(\omega))$. 
By the classical argument of Aubin \cite[Page 65]{Moussa2015}, we get the compactly embedded space $\mathcal{K}$ in $L(\Omega_T)$.
Then there again exists a subsequence of ($\omega^n$) denoted by ($\omega^n$) such that
\begin{equation*}
\begin{split}
& \omega^n \rightharpoonup  \omega^* \text { weakly in } L(\Omega_T) \text { and in } L^\infty\bigl(0, T ; L(\Omega)\bigr), \\
& \omega^n \longrightarrow  \omega^* \text { strongly in } L(\Omega_T), \\
& \omega^n \longrightarrow  \omega^* \text { a.e. in } L(\Omega_T),\\
& \omega^n(T) \longrightarrow  \omega^*(T) \text { in } L(\Omega_T).
\end{split}
\end{equation*}
Note that the space $D^\prime(\Omega_T)$ is the dual of $C_0^\infty(\Omega_T)$. If we let $\varphi\in C_0^\infty(\Omega_T)$, we get
\begin{equation*}
   \int_0^T \int_{\Omega} \omega_i^n \ ({ }^{\CFC}_{T} \Dat \varphi) \,dx dt 
    \longrightarrow \int_0^T \int_{\Omega} \omega_i^* \ ({ }^{\CFC}_{T} \Dat \varphi) \,dx dt,
\end{equation*}
and
\begin{equation*}
    \int_{\Omega} \varphi(T, x) \int_0^T \omega_i^n \,e^{-\gamma (T-t)} \,dt dx 
    \longrightarrow \int_{\Omega} \varphi(T, x) \int_0^T \omega_i^* \,e^{-\gamma (T-t)} \,dt dx.
\end{equation*}
By Lemma~\ref{L4}, we find
\begin{equation*} 
     { }^{\CFC} \Dat\omega_i^n \rightharpoonup { }^{\CFC} \Dat\omega_i^* \text{ weakly in } D^\prime(\Omega_T).
\end{equation*}
Writing $\omega_1^n \omega_2^n - \omega_1^* \omega_2^* = (\omega_1^n-\omega_1^*) \omega_2^n + \omega_1^*(\omega_2^n - \omega_2^*)$, 
using the convergence $\omega_i^n \longrightarrow \omega_i^*$ in $L^2(\Omega_T)$, 
and the boundedness of $(\omega_1^n), (\omega_2^n)$ in $L^{\infty}(\Omega_T)$, you get $\omega_1^n \omega_2^n \longrightarrow \omega_1^* \omega_2^*$ in $L^2(\Omega_T)$. 
We also have $u^n\to u^*$ in $L^2(\Omega_T)$ on a subsequence of $(u^n)$ denoted again by $(u^n)$. 
Using the closeness and convexity of $U_{ad}$ in $L^2(\Omega_T)$, we get that $U_{ad}$ is weakly closed. 
Then $u^*\in U_{ad}$ and as above $u^n \omega_1^n \longrightarrow u^* \omega_1^*$ in $L^2(\Omega_T)$. 
We also have $\cL\omega^n \rightharpoonup \chi $ weakly in $D^\prime(\Omega_T)$.

It remains to show that $\cL\omega^* = \chi$. Since $\cL$ is monotone, then
\begin{equation*}
    X_n = \langle \cL\omega^n - \cL v, \omega^n - v \rangle \ge 0,
    \quad \forall v\in D(\cL).
\end{equation*}
Recall that
\begin{equation*}
    \langle \cL\omega^n , \omega^n \rangle 
     = \langle { }^{\CFC} \Dat\omega^n - \Phi(\omega^n) , \omega^n \rangle.
\end{equation*}
Hence,
\begin{equation*}
  0 \le \langle\chi , \omega^*\rangle  
  - \langle\chi , v\rangle  
  - \langle\cL v , \omega^* - v\rangle 
  = \langle\chi - \cL v  , \omega^* - v\rangle .
\end{equation*}
Let $\delta>0$, we put $v = \omega^* - \delta h \in D(\cL)$. Then,
\begin{equation*}
       \delta \langle\chi - \cL(\omega^* - \delta h) , h\rangle \ge0.
\end{equation*}
Afterwards,
\begin{equation*}
   \langle\chi - \cL(\omega^* - \delta h),h\rangle \ge0.
\end{equation*}
Obviously, for $\delta \to 0$ we get
\begin{equation*}
    \langle\chi - \cL(\omega^*),h\rangle \ge0, \quad \forall h \in D(\cL).
\end{equation*}
Subsequently $\chi = \cL\omega^*$. By the uniqueness of the limit, we obtain that
\begin{equation*}
    \phi = { }^{\CFC} \Dat\omega^* - \cL\omega^*.
\end{equation*}
Now we may pass to the limit in the system satisfied by $\omega^n$ as $n \to \infty$, we find that an optimal solution of \eqref{E3.1}--\eqref{E3.5} is $(\omega^*, u^*)$.
\end{proof}

% ====================================================================
\section{Necessary optimality conditions}\label{S6}
Let $\omega^{\varepsilon}=(\omega_i^{\varepsilon})_{i=1,2,3,4}
=(\omega_1, \omega_2, \omega_3, \omega_4) (u^{\varepsilon})$ 
and $\omega^*=(\omega_i^*)_{i=1,2,3,4}=(\omega_1, \omega_2, \omega_3, \omega_4) (u^*)$ 
be the solutions of \eqref{E3.1}--\eqref{E3.3}, where $u^{\varepsilon}=u^*+\varepsilon u \in U_{ad}$, $\forall u \in U_{ad}$. 
We subtract the system associated to $\omega^*$ from the one corresponding to $\omega^{\varepsilon}$, 
where $\omega_{i}^{\varepsilon}=\omega_{i}^{*}+\varepsilon y_{i}^{\varepsilon}$, and get
\begin{equation}\label{E6.1}
\begin{cases}
   { }^{\CFC} \Dat y^{\varepsilon} =  \lambda\Delta y^{\varepsilon} + \frac{\Phi(\omega^{\varepsilon}) - \Phi(\omega^*)}{\varepsilon}, 
   &\text { in } \Omega_T \\ 
   \nabla y_{i}^\varepsilon\cdot \vec{n} = 0, \ i=1,2,3,4, &\text { on } \Sigma_T,\\
   y^{\varepsilon}(0, x)=0, &\text { in } \Omega.
\end{cases}
\end{equation}
On one side, a straightforward computation yields
\begin{equation*}
    \frac{\Phi(\omega^{\varepsilon}) - \Phi(\omega^*)}{\varepsilon} 
    = \cN_{\varepsilon} y^{\varepsilon} + Fu,
\end{equation*}
with 
\begin{equation*}
     \cN_{\varepsilon} = \cN(\omega_3^{\varepsilon}, u^{\varepsilon}) 
= \begin{pmatrix}
   \beta - \mu\omega_3^{\varepsilon}- \xi -u^{\varepsilon} & \beta & \beta - \mu\omega_1^* & \beta \\
\mu \omega_3^{\varepsilon} & - \xi - \mu & \mu\omega_1^* & 0\\
0 & \kappa & -\xi - \eta & 0\\
u^{\varepsilon} & 0 & \eta & -\xi
\end{pmatrix}
\ \text{ and } \
F=\begin{pmatrix}
-\omega_1^* \\0 \\0\\ \omega_1^*
\end{pmatrix}.
\end{equation*}
On the other hand, the elements of the matrix $\cN_{\varepsilon}$ are also uniformly bounded with respect to $\varepsilon$ and by Lemma~\ref{L4}, for $\varepsilon\to0$ in \eqref{E6.1}, we get
\begin{equation}\label{E6.2}
\begin{cases}
    { }^{\CFC} \Dat y = \lambda \Delta y + \cN y + Fu, &\text { in } \Omega_T \\ 
   \nabla y_i\cdot \vec{n} = 0, \quad i=1,2,3,4, &\text { on } \Sigma_T,\\
y(0,x) = 0, &\text { in } \Omega,
\end{cases}
\end{equation}
where $\cN = \cN(\omega_3^*, u^*)$. 
Using the same methodology as in % Section~\ref{S4} (Theorem~\ref{T1}), 
the proof of Theorem~\ref{T1},
we can determine that the problem~\eqref{E6.2} has one and only one solution.
To determine the adjoint problem associated to $y$, we introduce $\rho =(\rho_1, \rho_2, \rho_3, \rho_4)$ in such a way that
\begin{equation*}
   \int_0^T \int_{\Omega} \bigl({ }^{\CFC} \Dat y - \lambda \Delta y\bigr) \rho\, dx dt 
   = \int_0^T \int_{\Omega} \bigl(\cN y + Fu\bigr) \rho \,dx dt.
\end{equation*}
Since
\begin{equation*}
   \int_{\Omega}(\Delta y) \rho \,dx 
   = -\int_{\Omega} \nabla y \cdot \nabla \rho\, dx 
   =  \int_{\Omega} y (\Delta \rho) \,dx,
\end{equation*}
and by Lemma~\ref{L4}, we have
\begin{equation*}
    \int_0^T ({ }^{\CFC} \Dat y) \rho \,dt 
    = \int_0^T y  ({ }^{\CFC}_{T} \Dat \rho) \,dt 
    + \frac{1}{1-\alpha} \,\rho(T, x) \int_0^T y(t) \,e^{-\gamma (T-t)} \,dt.
\end{equation*}
Then the corresponding dual system for the system~\eqref{E3.1}--\eqref{E3.3} can be expressed as
\begin{equation}\label{E6.3}
\begin{cases}
   { }^{\CFC}_{T} \Dat \rho - \lambda \Delta \rho - \cN\rho  
   = \cW^* \cW \omega^*, &\text { in } \Omega_T \\ 
   \nabla \rho_{i}\cdot \vec{n} = 0, \quad i=1,2,3,4, &\text { on } \Sigma_T,\\
    \rho(T,x) = \cW^* \cW \omega^*(T,x), &\text { in } \Omega,
\end{cases}
\end{equation}
where $\cW$ is the matrix defined by
\begin{equation*}
\cW = \begin{pmatrix}
0 & 0 & 0 & 0\\
0 & 0 & 0 & 0\\
0 & 0 & 1 & 0 \\
0 & 0 & 0 & 0
\end{pmatrix}.
\end{equation*}
Using the same methodology as in Theorem~\ref{T1}, we can show that problem~\eqref{E6.3} has a unique solution.

\medskip
The following theorem gives us the necessary conditions for the optimal control $u^*$.
%%%%%%%%%%%%%%%%%%%%%%%%% Theorem 4
\begin{theorem}\label{T4}
Let $\omega^* = \omega(u^*)$ be an optimal solution of \eqref{E3.1}--\eqref{E3.5}. Then
\begin{equation}\label{E6.4}
   u^* =\max \biggl\{\min \Bigl(-\frac{1}{\sigma} F^* \rho, 1\Bigr), 0\biggr\} = \max \biggl\{\min \Bigl(1, -\frac{\omega_1^*}{\sigma}(\rho_1-\rho_4)\Bigr), 0\biggr\}.
\end{equation}
where $\rho$ is a solution of \eqref{E6.3}.
\end{theorem}

%%%%%%%%%%%%%%%%%%%%%%%%% Proof Theorem 4
\begin{proof}
Let $\omega^* = \omega(u^*)$ be a solution of \eqref{E3.1}--\eqref{E3.5}. We get
\begin{equation*}
\begin{split}
\cJ(\omega^*, u^*) 
&=\int_\Omega \bigl(\omega_3^*(T, x)\bigr)^2 \,dx 
    + \int_0^T \int_\Omega \bigl(\omega_3^*(t,x)\bigr)^2 \,dx dt 
   + \sigma \int_0^T \int_\Omega \bigl(u^*(t,x)\bigr)^2 \,dx dt\\
&= \|\omega_3^*(T, .)\|_{L^2(\Omega)}^2
   + \int_0^T \|\omega_3^*\|_{L^2(\Omega)}^2\, dt 
    + \sigma\|u^*\|_{L^2(\Omega_T)}^2 \\
&= \|\cW \omega^*(T, .)\|_{L(\Omega)}^2 
  + \int_0^T\|\cW \omega^*\|_{L(\Omega)}^2 \,dt 
   + \sigma\|u^*\|_{L^2(\Omega_T)}^2.
\end{split}
\end{equation*}
Let $\varepsilon>0$. 
Since the minimum of the objective functional is reached at $u^{*}$, then
\begin{equation*}
     \frac{\cJ(u^* + \varepsilon\nu) - \cJ(u)}{\varepsilon} \ge0,
\end{equation*}
is equivalent to
\begin{equation*}
    \int_0^T (F^* \rho + \sigma u^*,\nu)_{L^2(\Omega)} \,dt \ge0, \quad \forall \nu \in U_{ad}.
\end{equation*}
Since $L(\Omega) = \bigl(L^2(\Omega)\bigr)^4$ and $L^2(\Omega_T)$ are Hilbert spaces, then
\begin{equation*}
\begin{split}
    \cJ^{\prime}(\omega^*, u^*)(\nu)
    &= \lim _{\varepsilon\to0} \frac{1}{\varepsilon}\bigl(\cJ(\omega^{\varepsilon}, u^{\varepsilon}) - \cJ(\omega^*, u^*)\bigr) \\
    &= \lim _{\varepsilon\to0} \frac{1}{\varepsilon}\bigg(\int_0^T \int_{\Omega} \bigl((\omega_3^{\varepsilon})^2-(\omega_3^*)^2\bigr) \,dx dt 
    + \int_{\Omega} \Bigl(\bigl(\omega_3^{\varepsilon}(T, x)\bigr)^2-\bigl(\omega_3^*(T, x)\bigr)^2\Bigr) \,dx\\
    &\qquad +\sigma \int_0^T \int_{\Omega}\bigl((u^{\varepsilon})^2-(u^*)^2\bigr) \,dx dt\bigg) \\
    &= \lim_{\varepsilon\to0}\bigg(\int_0^T 
    \int_{\Omega}\Bigl(\frac{\omega_3^{\varepsilon}-\omega_3^*}{\varepsilon}\Bigr) (\omega_3^{\varepsilon}+\omega_3^*) \,dx dt 
    +\int_{\Omega}\Bigl(\frac{\omega_3^{\varepsilon}-\omega_3^*}{\varepsilon}\Bigr) (\omega_3^{\varepsilon}+\omega_3^*)(T,x) \,dx \\
     &\qquad +\sigma \int_0^T \int_{\Omega}(\varepsilon \nu^2 + 2\nu u^*) \,dx dt \bigg).
\end{split}
\end{equation*}
Because of $\omega_3^{\varepsilon} \to \omega_3^*$ in $L^2(Q)$ and $\omega_3^{\varepsilon}, \omega_3^* \in L^{\infty}(Q)$. We have
\begin{equation*}
\begin{split}
    \cJ^{\prime}(\omega^*, u^*)(\nu) 
    &= 2 \int_0^T \int_{\Omega} (\omega_3^*) \omega^{\prime} (u^*) \nu \, dx dt\\
&\qquad + 2 \int_{\Omega}\bigl((\omega_3^*) \omega^{\prime}(u^*) \nu\bigr)(T,x) \,dx 
    + 2 \sigma \int_0^T \int_{\Omega} \nu u^* \,dx dt.
\end{split}
\end{equation*}
This is the same as
\begin{equation*}
\begin{split}
   J^{\prime}(\omega^*, u^*)(\nu) 
   &= 2 \int_0^T \langle \cW \omega^*, \cW y \rangle \,dt \\
   &\qquad  + 2\langle \cW \omega^*(T,x), \cW y(T,x)\rangle  + 2 \sigma \int_0^T ( u^*, \nu)_{L^2(\Omega)} \,dt,
\end{split}
\end{equation*}
where $y = \omega^{\prime}(u^*) \nu$ is the unique solution of \eqref{E6.2}. Using \eqref{E6.2} and \eqref{E6.3}, we obtain 
\begin{equation*}
\begin{split}
    \int_0^T \langle \cW \omega^*, \cW y \rangle\,dt &
    + \langle \cW \omega^*(T, x), \cW y(T,x)\rangle \\
    &=\int_0^T \langle \cW^* \cW \omega^*, y\rangle \,dt 
    + \langle \cW \omega^*(T, x), \cW y(T, x)\rangle \\
   &= \int_0^T \langle { }^{\CFC}_{T} \Dat \rho - \lambda \Delta\rho - \cN\rho \ , \ y\rangle \,dt 
   + \langle \cW \omega^*(T, x), \cW y(T, x)\rangle \\
   &= \int_0^T\langle \rho \ , \ { }^{\CFC} \Dat y - \lambda \Delta y - \cN y\rangle \,dt\\
   &= \int_0^T \langle \rho, F\nu\rangle \,dt 
   = \int_0^T (F^* \rho, \nu)_{L^2(\Omega)} \,dt.
\end{split}
\end{equation*}
Finally, we have
\begin{equation*}
    \cJ^{\prime}(\omega^*, u^*)(\nu) 
    = 2\int_0^T  (F^* \rho + \sigma u^*, \nu)_{L^2(\Omega)} \,dt.
\end{equation*}
Since $U_{ad}$ is convex, %and by the G\^ateaux differentiability of $\cJ$ at $u^*$, 
we have $\cJ^{\prime}(\omega^*, u^*)(h-u^*)\ge0$, $\forall h\in U_{ad}$.  Therefore, we can say that %By standard arguments varying $h$, we get
\begin{equation*}
     u^{*} =\max \biggl\{\min \Bigl(-\frac{1}{\sigma} F^* \rho, 1\Bigr), 0\biggr\}.
\end{equation*}
\end{proof}

% ====================================================================
\section{Numerical results}\label{S7}
Here we give numerical approximations of the proposed fractional epidemiological model~\eqref{E3.1}--\eqref{E3.3}. 
Therefore, in the presence or absence of a vaccination program, we study the effect of the $\alpha$-order derivation of infection over 60 days. 
In the following, we assume that the disease is born in the subdomain $\Omega_0 = cell(8, 8)$ (the center of $\Omega$), where $\Omega$ is a city for the considered citizens with an area of 16\,km $\times$ 16\,km square grid. 
We have consulted \cite{SidiAmmi2022, SidiAmmi2023} (see page 12 and page 10, respectively) for the determination of certain parameters and initial conditions, comprehensively presented in Table~\ref{Tab2}, with the condition $\beta \le\xi$ as required in the proof of Theorem~\ref{T1}.

%%%%%%%%%%%%%%%%%%%%%%%%% Table 2
\begin{table}[hbtp]
\centering
\setlength{\tabcolsep}{0.3cm}
\caption{Parameter values and initial conditions for the SEIR model.}\label{Tab2}
\begin{tabular}{|c|c|c|c|}
\hline Description & Symbol & Value & Unit\\
\hline\hline
Initial susceptible people & $S_0$ & 100 in $\mathcal{C}^{\Omega_0}_\Omega$ and 70 in $\Omega_0$ & people$\cdot$ km$^{-2}$\\
\hline
Initial exposed people & $E_0$ & 0 in $\mathcal{C}^{\Omega_0}_\Omega$ and 20 in $\Omega_0$ & people$\cdot$ km$^{-2}$\\
\hline
Initial infected people & $I_0$ & 0 in $\mathcal{C}^{\Omega_0}_\Omega$ and 10 in $\Omega_0$ & people$\cdot$ km$^{-2}$\\
\hline
Initial removed people & $R_0$ & 0 in $\Omega$ & people$\cdot$ km$^{-2}$\\
\hline
Birth rate & $\beta$ & $0.02$ & day$^{-1}$\\
\hline
Diffusion coefficients & ${\lambda_{i} \ }_{(i=1,2,3,4)}$ & $0.1$ & km$^2\cdot$ day$^{-1}$\\
\hline
Disease transmission rate & $\kappa$ & $0.09$ & day$^{-1}$\\
\hline
Effective contact rate & $\mu$ & $0.05$ & (people$\cdot$ day)$^{-1}$km$^2$\\
\hline
Final time & $T$ & $60$ & day\\
\hline
Natural mortality rate & $\xi$ & $0.03$ & day$^{-1}$\\
\hline
Recovery rate & $\eta$ & $0.04$ & day$^{-1}$\\
\hline
\end{tabular}
\end{table}

% ====================================
\subsection{Forward-Backward sweep method algorithm}\label{S7.1}
To solve the proposed fractional SEIR model~\eqref{E3.1}--\eqref{E3.5}, we have used an explicit finite difference method implemented in MatLab to approximate the left-right $\mathcal{CFC}$ fractional derivatives.
Using a forward-time approach, we solved the state model~\eqref{E3.1}, while a backward-time approach was used to solve the dual problem~\eqref{E6.3}, guided by the transversality conditions. %based on the transversality conditions.

The initial value of the "\texttt{while}" loop condition is set to $Err=-1$, representing the minimum relative errors of $S, E, I, R, \rho_1, \rho_2, \rho_3, \rho_4$, and $u$, where the tolerance value is $10^{-3}$.
A uniform subdivision $\{x_i = 1 + i\delta_x / i=0,\dots,N_x - 1\}$ is used, where $N_x$ is the number of steps and $\delta_x$ is the step size. 
The number of time steps is denoted by $N_t$, and the step time is denoted by $\delta_t$.
The organigram of the algorithm can be summarized in Figure~\ref{F2}.
%%%%%%%%%%%%%%%%%%%%%%%%% Algorithm organigram
\begin{figure}[hbtp]
\centering
\begin{tikzpicture}[node distance=2cm]
\node (St) [rectangle, rounded corners, minimum width=1cm, minimum height=0.8cm,text centered, draw=black, fill=green!40] {Start};
% ============
\node (Init) [trapezium, trapezium stretches=true, trapezium left angle=70, trapezium right angle=110, minimum width=1cm, minimum height=0.8cm, text centered, draw=black, fill=blue!30, right of=St, xshift=1cm] {Initialization};
% ============
\node (ErTe) [circle, minimum width=1cm, minimum height=0.8cm, text centered, draw=black, fill=red!30, right of=Init, xshift=1cm] {$Err<0$};
% {$Err\footnote{}<0$};
% ============
\node (A) [rectangle, minimum width=1cm, minimum height=1cm, text centered, text width=2.9cm, draw=black, fill=orange!30, right of=ErTe, xshift=1.3cm, yshift=-1cm] {Solve the system~\eqref{E3.1}--\eqref{E3.3}};
% ============
\node (B) [rectangle, minimum width=2cm, minimum height=1cm, text centered, text width=2.3cm, draw=black, fill=orange!30, right of=A, xshift=1.5cm] {Solve the problem~\eqref{E6.3}};
% ============
\node (C) [rectangle, minimum width=2.5cm, minimum height=1cm, text centered, text width=2cm, draw=black, fill=orange!30, above of=B, yshift=0.5cm] {Using~\eqref{E6.4} update $u$};
% ============
\node (D) [rectangle, minimum width=2.5cm, minimum height=1cm, text centered, text width=2.6cm, draw=black, fill=orange!30, left of=C, xshift=-2cm] {Update the condition loop};
% ============
\node (FinSol) [trapezium, trapezium stretches=true, trapezium left angle=70, trapezium right angle=110, minimum width=1cm, minimum height=0.8cm, text centered, draw=black, fill=blue!30, below of=Init] {Final Solution};
% ============
\node (End) [rectangle, rounded corners, minimum width=1cm, minimum height=0.8cm,text centered, draw=black, fill=green!30, left of=FinSol, xshift=-1cm] {End};
% ====================================
% ====================================
\draw [thick,->,>=stealth] (St) -- ++ (Init);
\draw [thick,->,>=stealth] (Init) -- ++ (ErTe);
\draw [thick,->,>=stealth] (ErTe.east) -| (A) node[midway,above]{Yes};
\draw [thick,->,>=stealth] (A) -- ++ (B);
\draw [thick,->,>=stealth] (B) -- ++ (C);
\draw [thick,->,>=stealth] (C) -- ++ (D);
\draw [thick,->,>=stealth] (D) -| (ErTe); 
\draw [thick,->,>=stealth] (ErTe) |- (FinSol) node[midway,right]{No};
\draw [thick,->,>=stealth] (FinSol) -- ++ (End);
\end{tikzpicture}
\caption{Algorithm organigram for the proposed SEIR epidemic model.}\label{F2}
\end{figure}
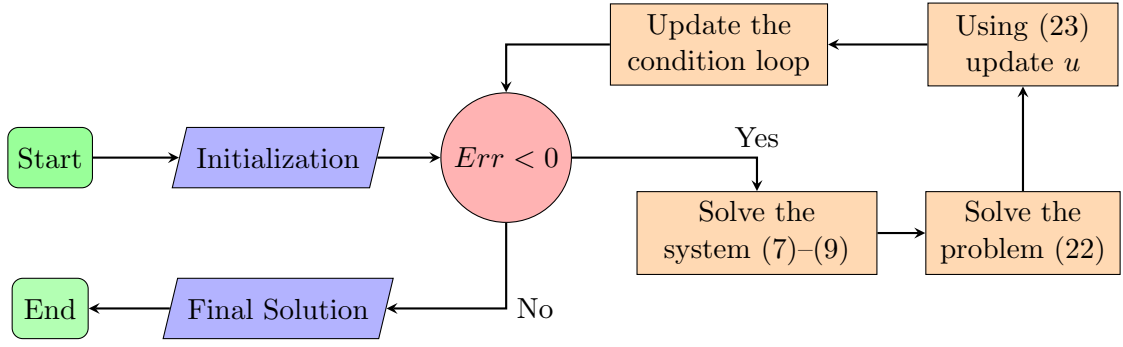
%\footnotetext[6]{The initial value of the condition loop is $Err=-1$. It is the minimum of the relative errors of $S,I,R,\rho_1,\rho_2,\rho_3$ and $u$ where the tolerance value is $10^{-3}$.} 

% ====================================
\subsection{Numerical approximations}\label{S7.2}
With different values of $\alpha$, the numerical results in the absence of vaccination are shown in Figures~\ref{F3}--\ref{F5}, and its presence in Figures~\ref{F6}--\ref{F8}. 

% =======================
\subsection*{Absence of vaccination}
In the initial scenario shown in Figure~\ref{F3}, when $\alpha = 1$, the epidemic takes 60 days to reach all areas within $\Omega$. 
However, the subsequent scenarios shown in Figures~\ref{F4} and \ref{F5} show that for $\alpha = 0.9$ and $\alpha = 0.8$, the spread of the pandemic extends beyond 60 days to cover the entire region.

\begin{figure}[hbtp]
\centering
\includegraphics[scale=0.29]{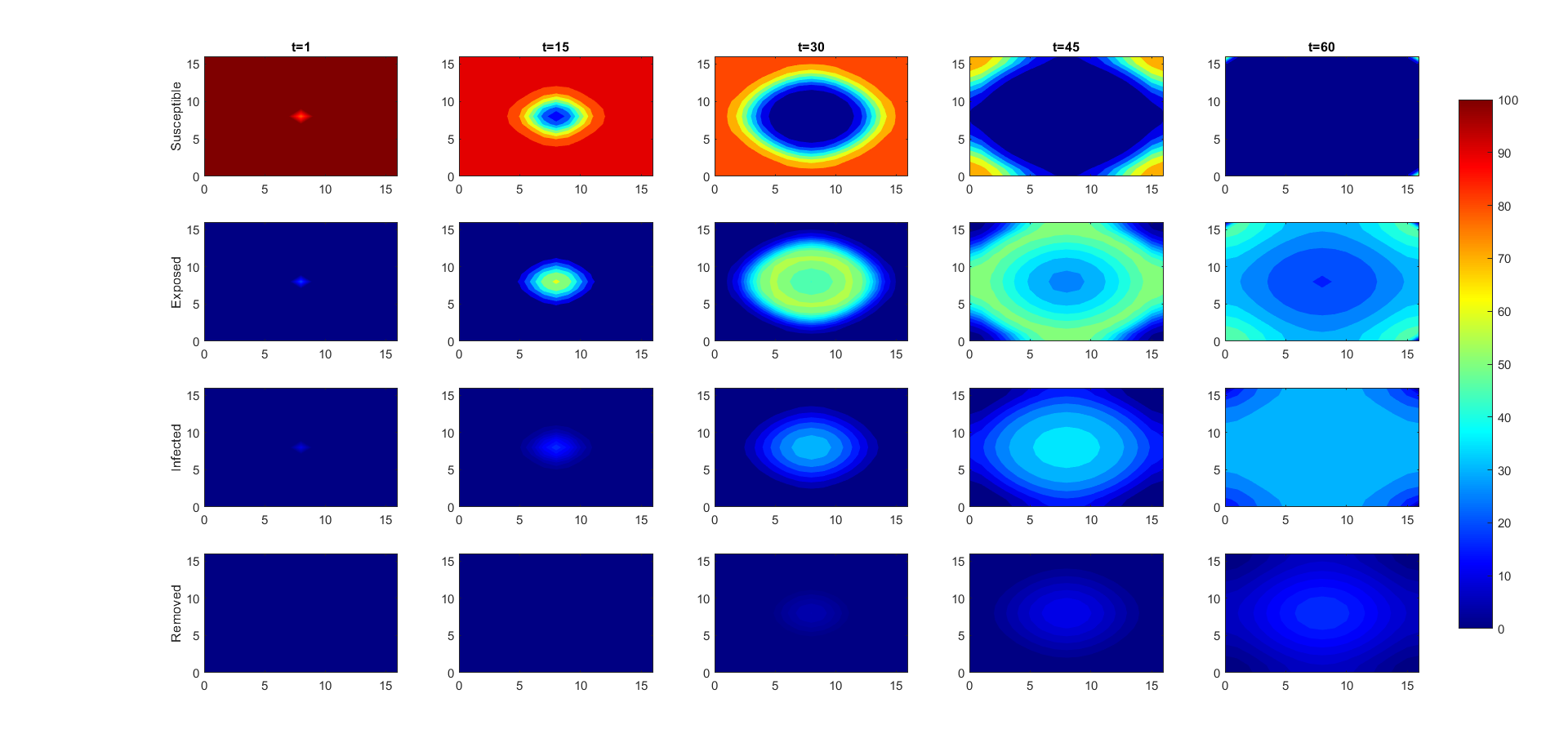}
\caption{Numerical approximations without vaccination strategy for $\alpha = 1$.}\label{F3}
\end{figure}

\begin{figure}[hbtp]
\centering
\includegraphics[scale=0.29]{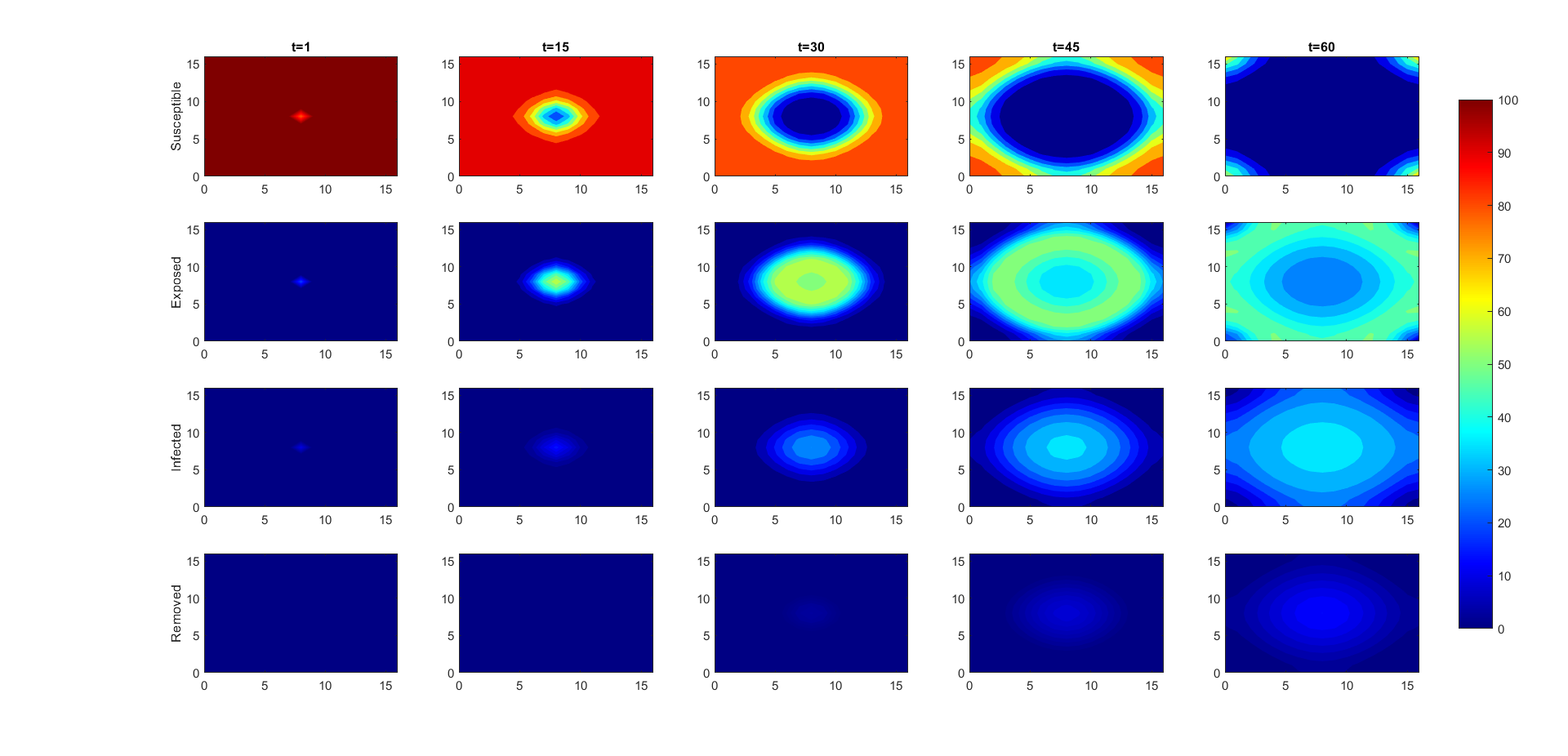}
\caption{Numerical approximations without vaccination strategy for $\alpha = 0.9$.}\label{F4}
\end{figure}

\begin{figure}[hbtp]
\centering
\includegraphics[scale=0.29]{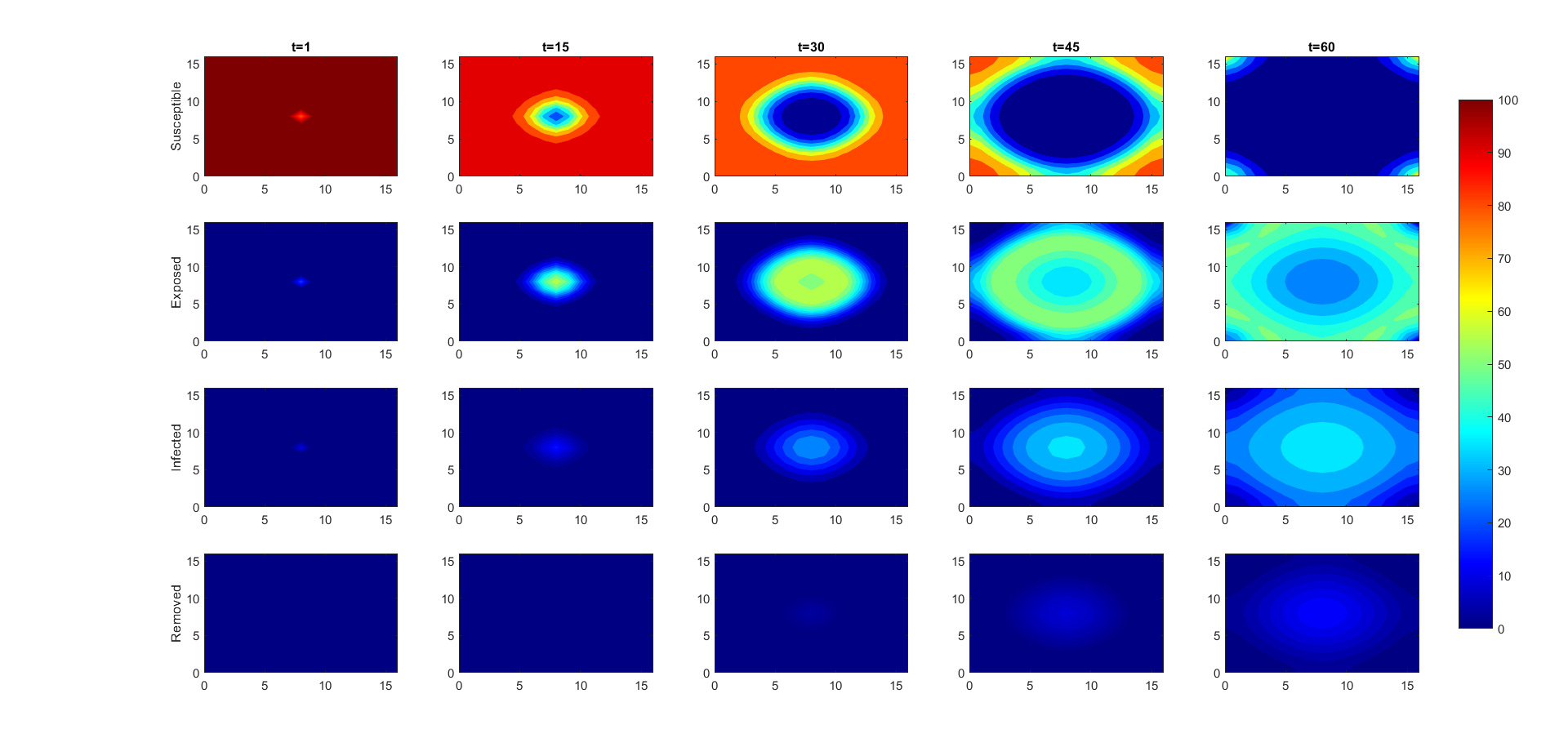}
\caption{Numerical approximations without vaccination strategy for $\alpha = 0.8$.}\label{F5}
\end{figure}

% ====================================
\subsection*{Presence of vaccination}
The second case shows an impressive transfer of vulnerable individuals to the recovered category.
However, the second case (Figures~\ref{F6}, \ref{F7}, and \ref{F8}) shows that the susceptible persons are transferred to the recovered category. 
This proves the effectiveness of the vaccination strategy in controlling the spread of the epidemic.

\begin{figure}[hbtp]
\centering
\includegraphics[scale=0.29]{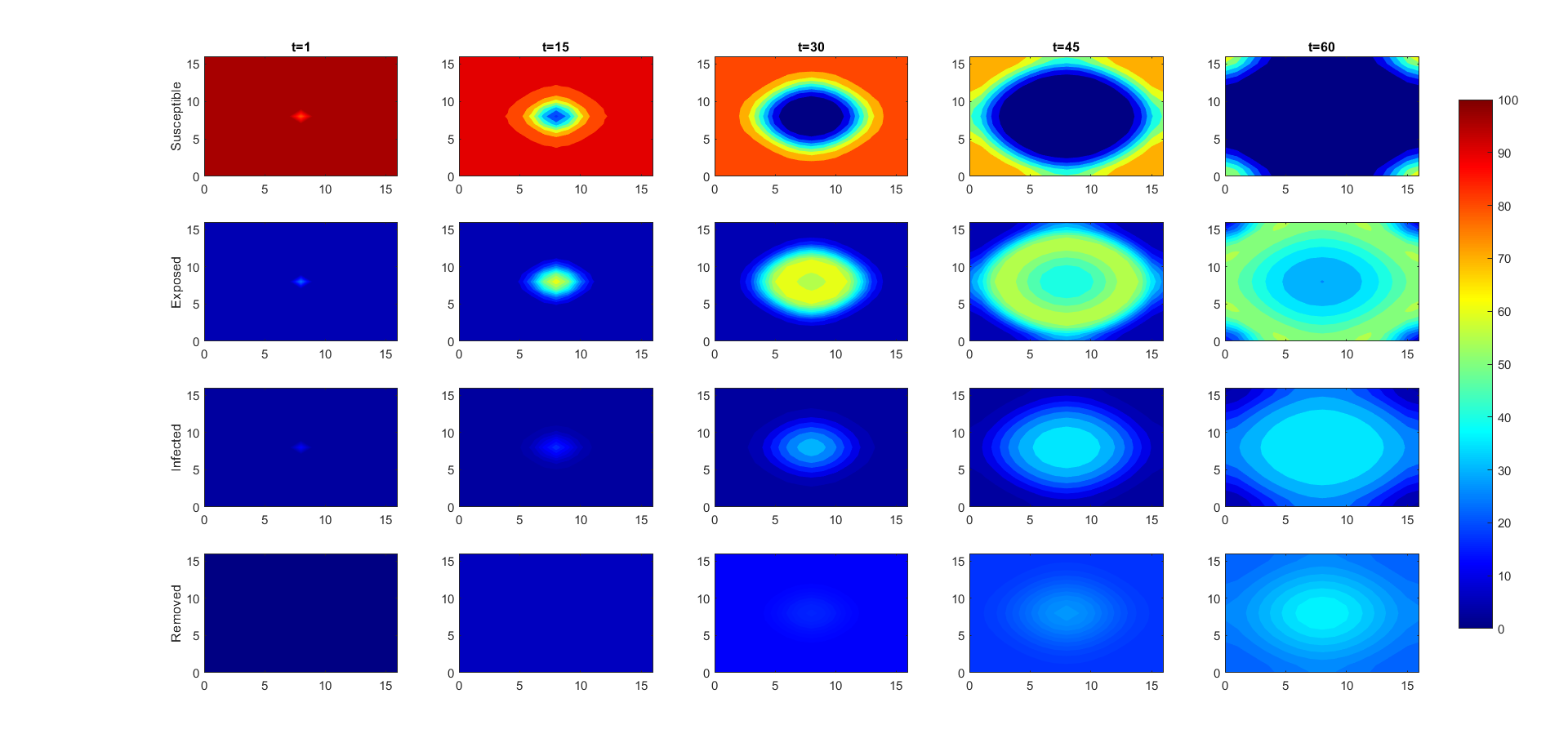}
\caption{Numerical approximations with vaccination strategy for $\alpha = 1$.}\label{F6}
\end{figure}

\begin{figure}[hbtp]
\centering
\includegraphics[scale=0.29]{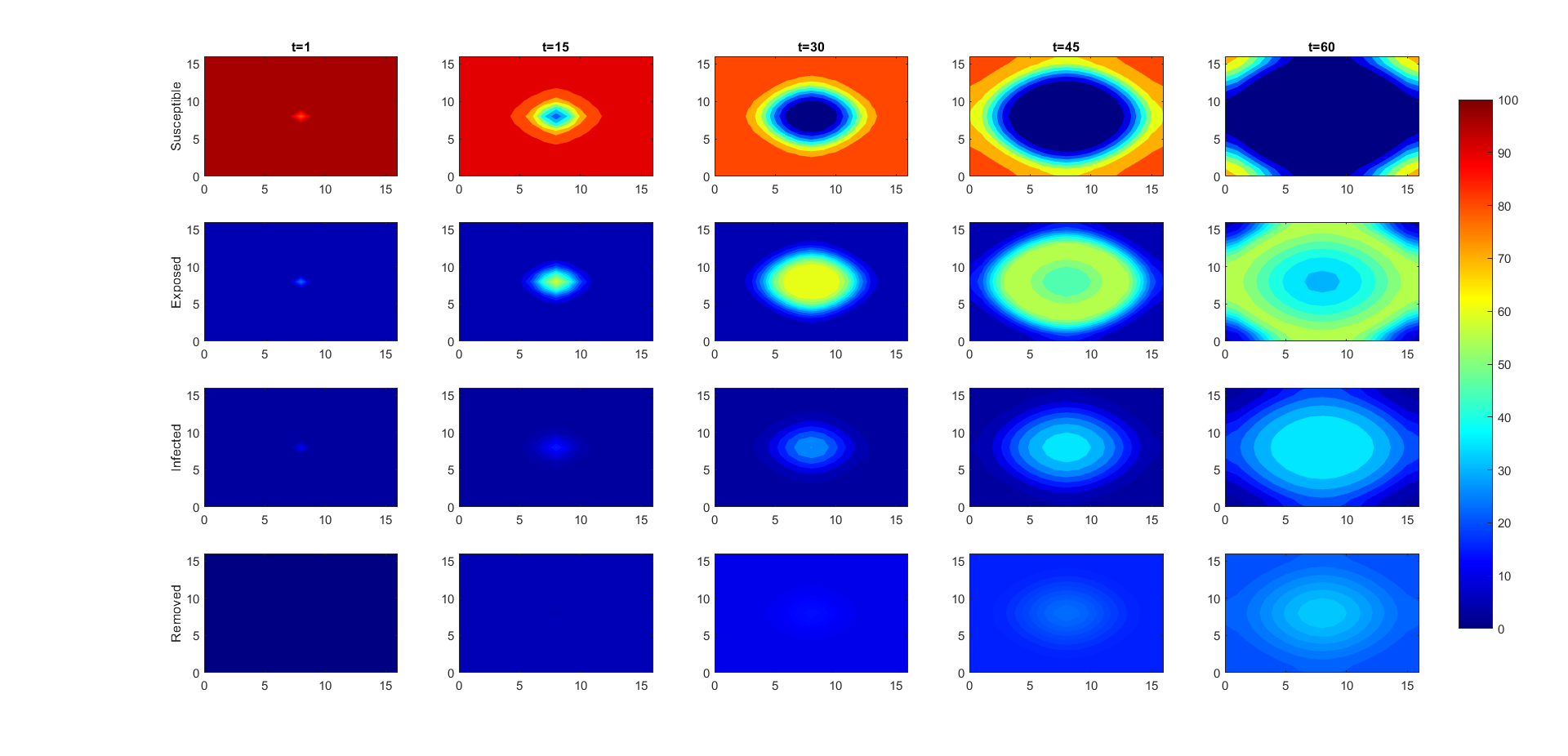}
\caption{Numerical approximations with vaccination strategy for $\alpha = 0.9$.}\label{F7}
\end{figure}

\begin{figure}[h]
\centering
\includegraphics[scale=0.29]{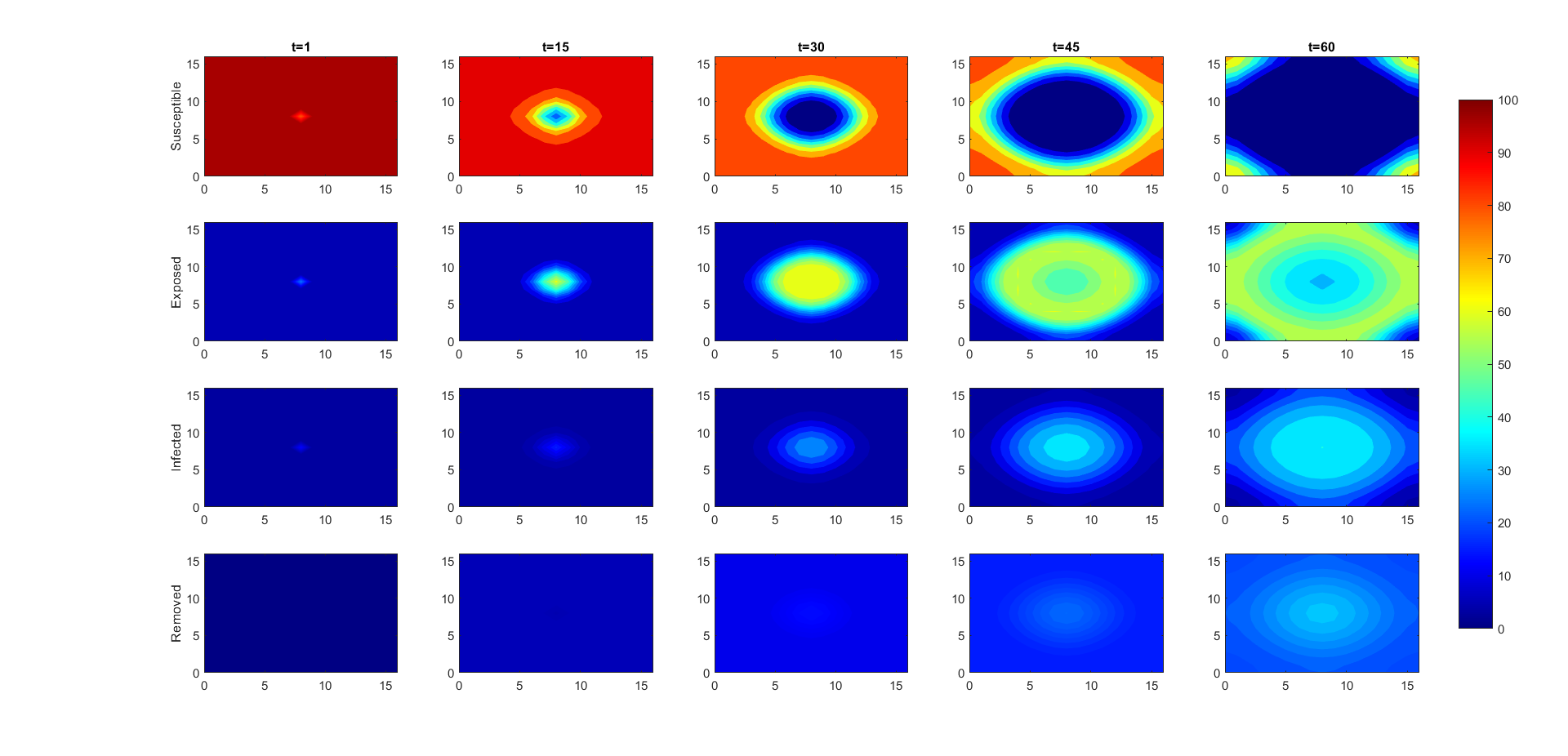}
\caption{Numerical approximations with vaccination strategy for $\alpha = 0.8$.}\label{F8}
\end{figure}

This section concludes with videos showing the spread of the epidemic over a 60-day period in the absence and presence of vaccination. 
The results are shown for different values of the parameter $\alpha$,
which is the order of the $\CFC$ fractional time derivative of the SEIR model~\eqref{E3.1}--\eqref{E3.5}.
\begin{itemize}
\item For $\alpha = 1$ (\href{https://drive.google.com/file/d/1yBiBc03bEGUk8PyOyRNKZM15t1zvJrKF/view?usp=drive_link}{Click here}).

\item For $\alpha = 0.9$ (\href{https://drive.google.com/file/d/1omTX9FsrX3atEMnK7tu29pRENnb2I1lZ/view?usp=drive_link}{Click here}).

\item For $\alpha = 0.8$ (\href{https://drive.google.com/file/d/1xfwRh1UQfjwXoe_7NgU16AfoPJx5v085/view?usp=drive_link}{Click here}).
\end{itemize}

% ====================================================================
\section{Conclusion}\label{S8}
In this article, we have introduced a novel application of optimal control theory to spatiotemporal models. 
The interactions between the four compartments, susceptible, exposed, infectious, and recovered, are modeled by a system of fractional equations using the Laplacian diffusion operator and the $\CFC$ time-fractional derivative. 
With our modest knowledge, this study can lead to more realistic models of spread in certain scenarios where we have the existence of a unique solution to our biological system as well as optimal control. 
In addition, optimal control is described through the appropriate use of state and adjoint variables. Furthermore, we have performed a comparative analysis of our system dynamics. 

The results show that when $\alpha$ does not take a fractional value, it leads to a rapid spread of the disease. 
Conversely, when $\alpha$ takes a fractional value, the disease takes more than 60 days to envelop all $\Omega$, incurring the same cost as the vaccination program in the case of natural order derivatives. 
Notably, our results indicate that the implementation of the vaccination strategy played a critical role in effectively controlling the spread of infection.

Future work will be concerned with considering more complex models than the SEIR model~\eqref{E3.1}--\eqref{E3.5}, by including more compartments and introducing e.g.\ waning effects of vaccination, time delays, and uncertainties. 
For the numerical solution, we will design nonstandard finite difference schemes (NSFDs) in the spirit of \cite{Maamar2024} that are capable of preserving the positivity of the solution, cf.\ Theorem~\ref{T2}, and other qualitative properties.

% ====================================================================

\section*{Declarations}

%\subsection*{Acknowledgments}
%The authors express their appreciation to the reviewers for their valuable, constructive comments, and suggestions.

% ====================================================================
\subsection*{CRediT author statement} 

\textit{A. Zinihi:} Conceptualization, 
Methodology, 
Software, 
Formal analysis, 
Investigation, 
Writing -- Original Draft, 
Writing -- Review \& Editing, 
Visualization. 

\textit{M. R. Sidi Ammi:} Conceptualization, 
Methodology, 
Validation, 
Formal analysis, 
Investigation,  
Writing -- Original Draft, 
Writing -- Review \& Editing,  
Supervision.

\textit{M. Ehrhardt:} Validation, 
Formal analysis, 
Investigation, 
Writing -- Original Draft, 
Writing -- Review \& Editing, 
Project administration.

% ====================================================================

%\subsection*{Funding} 
%The authors claim that the preparation of this manuscript execution received no financing.

% ====================================================================

\subsection*{Data availability} 
All information analyzed or generated, which would support the results of this work are available in this article.
No data was used for the research described in the article.

% =================================================================
\subsection*{Conflict of interest} 
The authors declare that there are no problems or conflicts 
of interest between them that may affect the study in this paper.

% =================================================================
\bibliographystyle{acm}
\bibliography{paper}

% =================================================================

\end{document}